\documentclass[10pt]{article}
\usepackage{amssymb, xcolor}
\usepackage{amsmath, amsfonts, amsthm, amssymb, verbatim,a4wide}   
\theoremstyle{plain}
        \newtheorem{theorem}{Theorem}[section]
\newtheorem{definition}[theorem]{Definition}
        
        \newtheorem{lemma}[theorem]{Lemma}
        
        \newtheorem{remark}[theorem]{Remark}

\numberwithin{equation}{section}

\newcommand \AAA A 
 
\newcommand \la \langle
\newcommand \ra \rangle 
\newcommand{\R}{\mathbb R} 
\newcommand \be         {\begin{equation}}
\newcommand \ee         {\varepsilon} 
\newcommand \del        \partial
\newcommand \eps    \varepsilon

\newcommand \lam \lambda

\newcommand{\weaks}{\stackrel{*}{\rightharpoonup}}   
\let\oldmarginpar\marginpar
\renewcommand\marginpar[1]{\-\oldmarginpar[\raggedleft\footnotesize #1]%
{\raggedright\footnotesize #1}}
\begin{document} 

\title{Initial-boundary value problems for nearly incompressible vector fields, and applications to the Keyfitz and Kranzer system} 
\author{Anupam Pal Choudhury\footnote{APC: Departement Mathematik und Informatik, Universit\"at Basel, Spiegelgasse 1, CH-4051 Basel, Switzerland. Email: anupampcmath@gmail.com}, Gianluca Crippa\footnote{GC: Departement Mathematik und Informatik, Universit\"at Basel, Spiegelgasse 1, CH-4051 Basel, Switzerland. Email: gianluca.crippa@unibas.ch}, Laura V. Spinolo\footnote{LVS: IMATI-CNR, via Ferrata 1, I-27100 Pavia, Italy. Email:  spinolo@imati.cnr.it}   
}  
\date{}  
\maketitle

\begin{abstract}  
We establish existence and uniqueness results for initial boundary value problems with nearly incompressible vector fields. We then apply our results to establish well-posedness
of the initial-boundary value problem for the Keyfitz and Kranzer system of conservation laws in several space dimensions.
\end{abstract}


 \section{Introduction} 
 The Keyfitz and Kranzer system is a system of conservation laws in several space dimensions that was introduced in~\cite{KK} and takes the form 
\begin{equation}
\label{e:KK}
 \partial_{t} U+\sum_{i=1}^{d} \partial_{x_{i}} (f^{i}(\vert U \vert) U) =0.
 \notag
 \end{equation} 
 The unknown is $U: \R^d \to \R^N$ and $|U|$ denotes its modulus. Also, for every $i=1, \dots, d$ the function $f^i: \R \to \R^N$ is smooth. In this work we establish existence and uniqueness results for the initial-boundary value problem associated to~\eqref{e:KK}. 
 
The well-posedness of the Cauchy problem associated to~\eqref{e:KK} was established by Ambrosio, Bouchut and De Lellis in~\cite{ABD,AD} by relying on a strategy suggested by Bressan in~\cite{Br}. Note that the results in~\cite{ABD,AD} are one of the very few well-posedness results that apply to systems of conservation laws in several spaces dimensions. Indeed, establishing either existence or uniqueness for a general system of conservation laws  in several space dimensions is presently a completely open problem, see~\cite{Daf,Serre1,Serre2} for an extended discussion on this topic. 

The basic idea underpinning the argument in~\cite{ABD,AD} is that~\eqref{e:KK} can be (formally) written as the coupling between a \emph{scalar} conservation law and a transport equation with very irregular coefficients. The scalar conservation law is solved by using the foundamental work by Kru{\v{z}}kov~\cite{Kr}, while the transport equation is handled by relying on Ambrosio's celebrated  
extension of the DiPerna-Lions' well-posendess theory, see~\cite{A} and~\cite{Diperna-Lions}, respectively, and~\cite{AC,Delellis2} for an overview.  Note, however, that Ambrosio's theory~\cite{A} does not directly apply 
to~\eqref{e:KK} owing to a lack of control on the divergence of the vector fields. In order to tackle this issue, a theory of \textit{nearly incompressible vector fields} was developed, see~\cite{Delellis1} for an extended discussion.  Since we will need it in the following, we recall the definition here.  
\begin{definition}\label{near-incom}
  Let $\Omega \subseteq \mathbb{R}^{d} $ be an open set and $T>0 $. We say that a vector field
  $b \in L^{\infty}((0,T) \times \Omega; \mathbb{R}^{d})$ is \textbf{nearly incompressible} if there are a density function $\rho \in L^{\infty}((0,T) \times \Omega) $  and a constant $C > 0 $ such that
  \begin{itemize}
  \item[i.] $0 \leq \rho \leq C, \ \mathcal{L}^{d+1}-a.e. \ \text{in} \ (0,T) \times \Omega,    $ and
  \item[ii.] the equation
  \begin{equation}
  \label{e:continuityrho}
  \partial_{t}\rho +\mathrm{div}(\rho b)=0
  \end{equation}
 holds in the sense of distributions in $(0, T) \times \Omega$.
  \end{itemize} 
 \end{definition}
The analysis in~\cite{ABD, AD, Delellis1} ensures that, if $b \in L^\infty ((0, T) \times \R^d; \R^d) \cap BV ((0, T) \times \R^d; \R^d)$ is a nearly incompressible vector field with density $\rho \in BV ((0, T) \times \R^d)$, then the Cauchy problem 
$$
\left\{
\begin{array}{ll}
         \partial_{t}[ \rho u] +\mathrm{div}[ \rho bu ]=0 &  
         \text{in $(0, T) \times \R^d$}\\
         u  = \overline{u} &  
         \text{at $t=0$}\\
\end{array} 
\right.
$$
is well-posed for every initial datum $\overline{u} \in L^\infty (\R^d)$. This result is pivotal to the proof of the well-posedness of the Cauchy problem for the Keyfitz and Kranzer system~\eqref{e:KK}. See also~\cite{ACFS} for applications of nearly incompressible vector fields to the so-called chromatography system of conservation laws. Note, furthermore, that here and in the following we denote by $BV$ the space of functions with \emph{bounded variation}, see~\cite{AFP} for an extended introduction. 

The present paper aims at extending the analysis in~\cite{ABD, AD, Delellis1}  to the case of initial-boundary value problems.  First, we establish the well-posedness of initial-boundary value problems with $BV$, nearly incompressible 
vector fields, see Theorem~\ref{IBVP-NC} below for the precise statement. In doing so, we rely on well-posedness results for continuity and transport equations with weakly differentiable vector fields established in~\cite{CDS1}, see also~\cite{CDS2} for related results.  Next, we discuss the applications to the Keyfitz and Kranzer system~\eqref{e:KK}. 

We now provide a more precise description of our results concerning nearly incompressible vector fields. We fix an open, bounded set $\Omega$ and a nearly incompressible vector field $b$ with density $\rho$ and we consider the initial-boundary value problem
 \begin{equation}
 \left\{
 \begin{array}{ll}
 \partial_{t} [\rho u]+ \text{div}[\rho u b]=0 & 
 \text{in $(0,T)\times \Omega$}\\
 u=\overline{u} & \text{at $t=0$}\\
 u= \overline{g} & \text{on $ \Gamma^{-}$},
  \end{array}
  \right.
 \label{prob-2}
 \end{equation} 
 where $\Gamma^{-}$ is the  part of the boundary $(0,T) \times \partial \Omega $ where the characteristic lines of the vector field $\rho b $ are \emph{inward pointing}. Note that, in general, if $b$ and $\rho$ are only weakly differentiable, one cannot expect that the solution $u$ is a regular function. Since $\Gamma^-$ will in general be negligible, then assigning the value of $u$ on $\Gamma^-$ is in general not possible. In~\S~\ref{s:formu} we provide the rigorous (distributional) formulation of the initial-boundary value problem~\eqref{prob-2} by relying on the theory of normal traces for low regularity vector fields, see~\cite{ACM,Anz,CF,CTZ}.    
 
 We can now state our well-posedness result concerning~\eqref{prob-2}. 
 \begin{theorem}\label{IBVP-NC}
 Let $T > 0 $ and $\Omega \subseteq \mathbb{R}^{d}$ be an open, bounded set with $C^2 $ boundary. Also, let  $b \in BV((0,T) \times \Omega; \mathbb{R}^{d}) \cap L^{\infty}((0,T) \times \Omega; \mathbb{R}^{d})$ be a nearly incompressible vector field with density $\rho \in BV((0,T) \times \Omega) \cap L^{\infty}((0,T) \times \Omega)$, see Definition~\ref{near-incom}. Further, assume that 
 $\overline{u} \in L^{\infty}(\Omega) $ and $\overline{g} \in L^{\infty}(\Gamma^{-}) $.
 
 Then there is a distributional solution $u \in L^{\infty}((0,T) \times \Omega) $ 
 to \eqref{prob-2} satisfying the maximum principle
 \begin{equation}
 \label{e:maxprin}
    \| u \|_{L^\infty} \leq \max \{ \| \overline u \|_{L^\infty}, \| \overline g \|_{L^\infty}  \}.
 \end{equation}
Also, if $u_1, \; u_2 \in L^\infty ((0, T) \times \Omega)$ are two different distributional solutions of the 
 same initial-boundary value problem, then $\rho u_1 = \rho u_2$
 $a.e.$ in $(0,T) \times \Omega$. 
 \end{theorem}
 Note that the reason why we do not exactly obtain uniqueness of the function $u$ is because $\rho$ can attain the value $0$. If $\rho$ is bounded away from $0$, then the distributional solution $u$ of~\eqref{prob-2} is unique.
 Also, we refer to~\cite{Bar,Boyer,CDS1,CDS2,GS,Mis} for related results on the well-posedness of initial-boundary value problems  for continuity and transport equation with weakly differentiable vector fields. 
 
 In~\S~\ref{s:KK} we discuss the applications of Theorem~\ref{IBVP-NC} to the Keyfitz and Kranzer system and our main well-posedness result is Theorem~\ref{t:KK}. Note that the proof of Theorem~\ref{t:KK} combines Theorem~\ref{IBVP-NC}, the analysis in~\cite{Delellis1}, and well-posedness results for the initial-boundary value problems for scalar conservation laws established in~\cite{BLN, CR, Serre2}. 
 
 \subsection*{Paper outline}
In~\S~\ref{s:prel} we go over some preliminary results concerning normal traces of weakly differentiable vector fields.  By relying on these results, in~\S~\ref{s:formu} we provide the rigorous formulation of the initial-boundary value problem~\eqref{prob-2}. In~\S~\ref{s:exi} we establish the existence part of Theorem~\ref{IBVP-NC}, and in~\S~\ref{s:uni} the uniqueness. 
In~\S~\ref{s:ssc} we establish some stability and space continuity property results. Finally, in~\S~\ref{s:KK} we discuss the applications to the Keyfitz and Kranzer system.   
 
\subsection*{Notation}
For the reader's convenience, we collect here the main notation used in the present paper. 

\begin{itemize}
\item $\mathrm{div}$: the divergence, computed with respect to the $x$ variable only.
\item $\mathrm{Div}$: the complete divergence, i.e. the divergence computed with respect to the $(t, x)$ variables. 
\item $\mathrm{Tr} ( B, \partial \Lambda)$: the normal trace of the bounded, measure-divergence vector field $B$ on the boundary of the set $\Lambda$, see \S~\ref{s:prel}.  
\item $(\rho u)_0$,  
$\rho_0$: the initial datum of the functions $\rho u$ and $\rho$, see Lemma~\ref{trace-existence} and Remark~\ref{r} .
\item $ T (f)$: the trace of the $BV$ function $f$, see Theorem~\ref{bv-trace}.
\item $\mathcal H^s$: the $s$-dimensional Hausdorff measure.
\item $f_{|_E}$: the restriction of the function $f$ to the set $E$. 
\item $\mu \llcorner E$: the restriction of the measure $\mu$ to the measurable set $E$. 
\item $a.e.$: almost everywhere. 
\item $|\mu|$: the total variation of the measure $\mu$.
\item $a \cdot b$: the (Euclidean) scalar product between $a$ and $b$.  
\item $\mathbf{1}_E:$ the characteristic function of the measurable set $E$. 
\item $\Gamma, \Gamma^-, \Gamma^+, \Gamma^0$: see~\eqref{e:gamma}. 
\item $\vec n$: the outward pointing, unit normal vector to $\Gamma$. 
\end{itemize}

 \section{Preliminary results}
 \label{s:prel}
 In this section, we briefly recall some notions and results that shall be used in the sequel. 
 
First, we discuss the notion of normal trace for weakly differentiable vector fields, see~\cite{ACM,Anz,CF,CTZ}. Our presentation here closely follows that of \cite{ACM}. Let $\Lambda \subseteq \mathbb{R}^{N} $ be an open set and let us denote by $\mathcal{M}_{\infty}(\Lambda) $, the family of bounded, measure-divergence vector fields. The space $\mathcal{M}_{\infty}(\Lambda) $, therefore, consists of bounded functions $B \in L^{\infty}(\Lambda;\mathbb{R}^{N})$ such that the distributional divergence of $B$ (denoted by $\text{Div}  B $) is a locally bounded Radon measure on $\Lambda$.

The normal trace of $B \in \mathcal{M}_{\infty}(\Lambda)$ on the boundary $\partial \Lambda $ can be defined as follows.
 \begin{definition}
 Let $\Lambda \subseteq \mathbb{R}^{N}$ be an open and bounded set with  Lipschitz continuous boundary and let $B \in \mathcal{M}_{\infty}(\Lambda)$. The normal trace of $B$ on $\partial \Lambda $ is a distribution 
defined by the identity
\begin{equation}
\Big\langle \emph{Tr}(B,\partial \Lambda), \psi \Big\rangle = \int_{\Lambda} \nabla \psi \cdot B \ dy + \int_{\Lambda} \psi\  d(\emph{Div} B) , \qquad \forall \ \psi \in C^{\infty}_{c}(\mathbb{R}^{N}).
\label{prel-1}
\end{equation}
Here $\emph{Div} B $ denotes the distributional divergence of $B$ and is a bounded Radon measure on $\Lambda $.  
\end{definition}
Note that, owing to the Gauss-Green formula, if $B$ is a smooth vector field, then $\mathrm{Tr}(B,\partial \Lambda) = B \cdot \vec n$,
 where $\vec n$ denotes the outward pointing, unit normal vector to $\partial \Lambda$.

Note, furthermore, that the analysis in~\cite{ACM} shows that the normal trace distribution satisfies the following properties. 
 \begin{itemize}
 \item[(a)] The normal trace distribution is induced by an $L^{\infty}$ function on $\partial \Lambda $, which we shall continue to refer to as $\mathrm{Tr}(B,\partial \Lambda) $. The bounded function $\mathrm{Tr}(B,\partial \Lambda) $ satisfies \[\Vert \mathrm{Tr}(B,\partial \Lambda) \Vert_{L^{\infty}(\partial \Lambda)} \leq \Vert B \Vert_{L^{\infty}(\Lambda)}. \] 
 
 \item[(b)] Let $\Sigma $ be a Borel set contained in $\partial \Lambda_{1} \cap \partial \Lambda_{2} $, and let $\vec{n}_{1}=\vec{n}_{2} \ \text{on}\ \Sigma$ (here $\vec{n}_{1},\vec{n}_{2} $ denote the outward pointing, unit normal vectors to $\partial \Lambda_{1},\partial \Lambda_{2} $ respectively). Then 
 \begin{equation}
 \mathrm{Tr}(B, \partial \Lambda_{1}) = \mathrm{Tr}(B, \partial \Lambda_{2}) \qquad \text{$\mathcal{H}^{N-1}$-a.e.~on $\Sigma$.}
 \label{prel-2} 
 \end{equation} 
 \end{itemize}
In the following we will use several times the following renormalization result, which was established in~\cite{ACM}. 
\begin{theorem}\label{trace-renorm}
Let $B \in BV (\Lambda;\mathbb{R}^{N})\cap L^{\infty}(\Lambda;\mathbb{R}^{N}) $ and $w \in L^{\infty}(\Lambda) $ be such that $\emph{Div} (wB )$ is a Radon measure. If $\Lambda' \subset \subset \Lambda $ is an open set with bounded and Lipschitz continuous boundary and $h \in C^{1}(\mathbb{R})$, then
\begin{equation}
\emph{Tr}(h(w)B,\partial \Lambda')=h\left(\frac{\emph{Tr}(wB,\partial \Lambda')}{\emph{Tr}(B,\partial \Lambda')}\right) \emph{Tr}(B,\partial \Lambda') \qquad \text{$\mathcal{H}^{N-1}$-a.e.~on~$\partial \Lambda'$,}
\notag
\end{equation}
where the ratio $\displaystyle{\frac{\emph{Tr}(wB,\partial \Lambda')}{\emph{Tr}(B,\partial \Lambda')} }$ is arbitrarily defined at points where the trace $\emph{Tr}(B,\partial \Lambda') $ vanishes. 
\end{theorem}   
We can now introduce the notion of normal trace on a general bounded, Lipschitz continuous, oriented hypersurface $\Sigma \subseteq \mathbb{R}^{N}$ in the following manner. Since $\Sigma $ is oriented, an orientation of the normal vector $\vec{n}_{\Sigma} $ is given. We can then find a domain $\Lambda_{1} \subseteq \mathbb{R}^{N} $ such that $\Sigma \subseteq \partial \Lambda_{1} $ and the normal vectors $\vec{n}_{\Sigma}, \vec{n}_{1} $ coincide. Using \eqref{prel-2}, we can then define
 \[\text{Tr}^{-}(B,\Sigma):= \text{Tr}(B,\partial \Lambda_{1}). \]
 Similarly, if $\Lambda_{2} \subseteq \mathbb{R}^{N} $ is an open set satisfying $\Sigma \subseteq \partial \Lambda_{2} $, and $\vec{n}_{2}=-\vec{n}_{\Sigma} $, we can define
 \[\text{Tr}^{+}(B,\Sigma):=- \text{Tr}(B,\partial \Lambda_{2}). \]
 Furthermore we have the formula
 \[(\text{Div} B)\llcorner \Sigma= \Big( \text{Tr}^{+}(B,\Sigma)-\text{Tr}^{-}(B,\Sigma) \Big) \mathcal{H}^{N-1} \llcorner \Sigma. \]
 Thus $\text{Tr}^{+} $ and $\text{Tr}^{-} $ coincide $\mathcal{H}^{N-1}- $a.e. on $\Sigma$ if and only if $\Sigma $ is a $(\text{Div} B)$-negligible set.\\  

 We next recall some results from \cite{ACM} concerning space continuity.
 \begin{definition}\label{graph}
 A family of oriented surfaces $\{\Sigma_{r} \}_{r \in I} \subseteq \mathbb{R}^{N} $ (where $I \subseteq \mathbb{R}$ is an open interval) is called a family of graphs if there
 exist 
 \begin{itemize}
 \item a bounded open set $D \subseteq \mathbb{R}^{N-1}$,
 \item a Lipschitz function $f:D \rightarrow \mathbb{R}$,
 \item a system of coordinates $(x_{1},\cdots,x_{N})$
 \end{itemize}
 such that the following holds true:
 For each $r \in I$, we can write
 \begin{equation}
 \Sigma_{r}=\big\{(x_{1},\cdots,x_{N}): f(x_{1},\cdots,x_{N-1})-x_{N}=r \big\},
 \label{ACM-99}
 \end{equation}
 and the orientation of $\Sigma_{r}$ is determined by the normal $\displaystyle{\frac{(-\nabla f,1)}{\sqrt{1+\vert \nabla f \vert^{2}}} }$.  
 \end{definition}
We now quote a space continuity result. 
 \begin{theorem}[see \cite{ACM}]\label{Weak-continuity}
 Let $B \in \mathcal{M}_{\infty}(\mathbb{R}^{N})$ and let $\{\Sigma_{r} \}_{r \in I} $ be a family of graphs as above. For a fixed $r_{0} \in I$, let us define the functions $\alpha_{0}, \alpha_{r}: D \rightarrow \mathbb{R} $ as
\begin{equation}
\begin{aligned}
\alpha_{0}(x_{1},\cdots,x_{N-1})&:=\emph{Tr}^{-}(B,\Sigma_{r_{0}})(x_{1},\cdots,x_{N-1},f(x_{1},\cdots,x_{N-1})-r_{0}), \ \text{and} \\
\alpha_{r}(x_{1},\cdots,x_{N-1})&:=\emph{Tr}^{+}(B,\Sigma_{r})(x_{1},\cdots,x_{N-1},f(x_{1},\cdots,x_{N-1})-r) .
\end{aligned}
\label{ACM-100}
\end{equation}  
 Then $\alpha_{r} \stackrel{*}{\rightharpoonup} \alpha_{0} $ weakly$^{*}$ in $L^{\infty}(D,\mathcal{L}^{N-1} \llcorner D) $ as $r \rightarrow r^{+}_{0}$. 
 \end{theorem}
We will also need the following result, which was originally established in~\cite{CDS1}.
 \begin{lemma}\label{extension}
  Let $\Lambda \subseteq \mathbb{R}^{N}$ be an open and bounded set with bounded and Lipschitz continuous boundary and let $B$ belong to
  $\mathcal{M}_{\infty}(\Lambda)$. Then the vector field 
  \begin{equation}
  \tilde{B}(z):=
\left\{\begin{array}{ll}
  B(z) & z \in \Lambda \\
    0  & \text{otherwise}
   \end{array}\right.
  \notag
  \end{equation}
belongs to $\mathcal{M}_{\infty}(\mathbb{R}^{N})$.
 \end{lemma}
We conclude by recalling some results concerning traces of $BV$ functions and we refer to~\cite[\S 3]{AFP} for a more extended discussion. 
\begin{theorem}
\label{bv-trace}
Let $\Lambda \subseteq \mathbb{R}^{N}$ be an open and bounded set with bounded and Lipschitz continuous boundary. There exists a bounded linear mapping
\begin{equation}
T: BV(\Lambda) \rightarrow L^{1}(\partial \Lambda;\mathcal{H}^{N-1}) 
\label{ACM-101}
\end{equation}
such that $T (f) = f_{|_{\partial \Lambda}}$ if $f$ is continuous up to the boundary. Also,
\begin{equation}
\int_{\Lambda} \nabla \psi \cdot f \ dy = - \int_{\Lambda} \psi \ d(\emph{\text{Div}} f) + \int_{\partial \Lambda} \psi  \ Tf \cdot \vec n \ d\mathcal{H}^{N-1},
\label{ACM-102}
\end{equation}
for all $f \in BV(\Lambda)$ and $\psi \in C^{\infty}_{c}(\mathbb{R}^{N})$. In the above expression, $\vec n$ denotes the outward pointing, unit normal vector to $\partial \Lambda$.
\end{theorem}
By comparing~\eqref{prel-1} and~\eqref{ACM-102} we conclude that
\begin{equation}
\label{e:equal}
        \mathrm{Tr} (f, \partial \Lambda) = T (f) \cdot \vec n, \quad \text{for every $f \in BV (\Lambda)$}.
\end{equation}
By combining Theorems~3.9 and~3.88 in~\cite{AFP} we get the following result.
\begin{theorem}[\cite{AFP}]
\label{t:traceafp}
Assume $\Lambda \subseteq \R^N$ is an open set with bounded and Lipschitz continuous  boundary. If $f \in BV (\Lambda; \R^m)$, then there is a sequence $\{\tilde f_m \} \subseteq C^\infty (\Lambda)$ such that
\begin{equation}
\label{e:tracefp}
    \tilde f_m \to f \;  \text{ strongly in $L^1 (\Lambda; \R^m)$}, 
    \qquad 
    T (\tilde f_m) \to T(f)  \text{ strongly in $L^1 (\partial \Lambda; \R^m)$}.  
\end{equation} 
Also, we can choose $\tilde f_m$ in such a way that 
\begin{itemize}
\item
$\tilde f_m \ge 0$ if $f \ge 0$,
\item if $f \in L^\infty (\Lambda; \R^m)$, then 
\begin{equation}
\label{e:four}
   \| \tilde f_m \|_{L^\infty} \leq  4 \| f \|_{L^\infty}. 
\end{equation}
\end{itemize}
\end{theorem}
A sketch of the proof of Theorem~\ref{t:traceafp} is provided in~\S~\ref{s:proof1}. 
\section{Distributional formulation of the problem}
\label{s:formu}
In this section, we follow~\cite{Boyer,CDS1} and we provide the distributional formulation of the problem \eqref{prob-2}. We first establish a preliminary result. 
 \begin{lemma}\label{trace-existence}
 Let $\Omega \subseteq \mathbb{R}^{d}$ be an open bounded set with $C^2$ boundary  
and let $T > 0 $. We assume that
 $b \in L^{\infty}((0,T) \times \Omega; \mathbb{R}^{d}) $ is a nearly incompressible vector field with density $\rho \in L^{\infty}((0,T) \times \Omega) $, see Definition \ref{near-incom}. If $u \in L^{\infty}((0,T) \times \Omega)$ satisfies
 \begin{equation}
 \int_{0}^{T} \int_{\Omega} \rho u (\partial_{t} \phi+b \cdot \nabla \phi) \ dx dt= 0, \quad \forall \ \phi \in \mathcal{C}^{\infty}_{c} ((0,T) \times \Omega),    
  \label{iden-2}
 \end{equation}
 then there are two unique functions, which we henceforth denote by $\emph{Tr}(\rho u b) \in L^{\infty}((0,T) \times \partial \Omega) $ and $(\rho u)_{0} \in L^{\infty}(\Omega)$, that satisfy
\begin{equation}
\int_{0}^{T} \! \!  \int_{\Omega} \rho u (\partial_{t} \psi+ b \cdot \nabla \psi) \ dx dt= \int_{0}^{T} \! \! \int_{\partial \Omega} \emph{Tr}(\rho u b) \psi \ 
d\mathcal{H}^{d-1}\ dt - \int_{\Omega} \psi(0,\cdot) (\rho u)_{0}\ dx, \quad  \forall \psi \in  \mathcal{C}^{\infty}_{c} ([0,T) \times \mathbb{R}^{d}).
 \label{iden-3}
\end{equation}
Also, we have the bounds
\begin{equation}
\label{e:maxprintraces2}
      \| \emph{Tr}(\rho u b) \|_{L^\infty((0,T) \times \partial \Omega) } , \;
     \|  (\rho u)_{0} \|_{ L^{\infty}(\Omega)}
      \leq  \max\{ \| \rho u  \|_{L^\infty((0,T) \times  \Omega) } ; \| \rho u b \|_{L^\infty((0,T) \times  \Omega) } \}.
\end{equation}
\end{lemma}
\begin{proof}
First of all, let us note that the uniqueness of such functions follow from the liberty in choosing the test functions $\psi$. Therefore
it is enough to discuss the existence of the functions with the above properties.
 Let us define 
 \begin{equation}
 B(t,x):= \left\{
      \begin{array}{ll}
           (u \rho, u \rho b) & (t,x) \in (0,T)\times \Omega \\
           0 &\text{elsewhere in}\ \mathbb{R}^{d+1}.\\
          \end{array} \right. 
  \label{e:extend}
 \end{equation}
Then $B \in L^{\infty}(\mathbb{R}^{d+1})$ and from \eqref{iden-2}, it also follows that $\big[\text{Div} B \llcorner {(0,T) \times \Omega} \big]=0 $.
We can now apply Lemma \ref{extension} with $\Lambda= (0,T) \times \Omega $ to conclude that $B \in \mathcal{M}_{\infty}(\mathbb{R}^{d+1}).$
Hence $B$ induces the existence of normal trace on $\partial \Lambda$. 
Let 
\begin{equation}
 \text{Tr}(\rho u b):= \text{Tr} (B,\partial \Lambda) \Big\vert_{(0,T) \times \partial \Omega}, \ \ 
 (\rho u)_{0}:= -\text{Tr}(B,\partial \Lambda) \Big\vert_{\{0 \} \times \Omega}.  
 \notag
 \end{equation}
 The identity \eqref{iden-3} then follows from \eqref{prel-1} by virtue of the fact that $\text{Div}B=0 \ \text{in}\ (0,T)\times \Omega $.
\end{proof}
\begin{remark}
\label{r}
We define the vector field $P:=(\rho,\rho b) $ and we point out that $P \in {L^{\infty}((0,T) \times \Omega; \mathbb{R}^{d+1})}$   since 
 $\rho$ and $b$ are both bounded functions. By introducing the same extension as in~\eqref{e:extend} and using the fact that
 \begin{equation}
 \int_{0}^{T} \int_{\Omega} \rho  (\partial_{t} \phi+b \cdot \nabla \phi) \ dx dt= 0, \quad \forall \ \phi \in \mathcal{C}^{\infty}_{c} ((0,T) \times \Omega),    
  \notag
 \end{equation}
 we can argue as in the proof of the above lemma to establish the existence of unique functions $\emph{Tr}(\rho  b) \in L^{\infty}((0,T) \times \partial \Omega) $ and 
 $\rho_0 \in L^\infty(\Omega)$ defined as 
 $$
    \emph{Tr}(\rho b):= \emph{Tr}(P, \partial \Lambda) \Big\vert_{(0,T) 
    \times \partial      \Omega}, \quad 
    \rho_0 : =  - \emph{Tr}(P, \partial \Lambda) \Big\vert_{\{ 0 \} \times 
          \Omega}.
 $$ 
 In this way, we can give a meaning to the normal trace $\mathrm{Tr} (\rho b)$ and to the initial datum $\rho_0$. Also, we have the bounds
 \begin{equation}
\label{e:maxprintraces1}
      \| \emph{Tr}(\rho b) \|_{L^\infty((0,T) \times \partial \Omega) } , \;
     \|  \rho_{0} \|_{ L^{\infty}(\Omega)}
      \leq  \max\{ \| \rho  \|_{L^\infty((0,T) \times  \Omega) } ; \| \rho  b \|_{L^\infty((0,T) \times  \Omega) }\}.
\end{equation}
\end{remark}
We can now introduce the distributional formulation to the problem \eqref{prob-2} by using Lemma \ref{trace-existence}.
We introduce the following notation:   
 \begin{equation}
 \left.
 \begin{array}{ll}
 \Gamma: = (0, T) \times \partial \Omega, 
  &  \Gamma^{-}:= \{(t,x) \in \Gamma: \ \text{Tr} (\rho b)(t,x)<0 \},\\
       \Gamma^{+}:=\{(t,x) \in \Gamma: \ \text{Tr} (\rho b)(t,x) > 0 \}, &
      \Gamma^0:=\{(t,x) \in \Gamma: \ \text{Tr} (\rho b)(t,x) = 0 \}.  \\
      \end{array}
      \right.
 \label{e:gamma}
 \end{equation}
\begin{definition}
\label{d:distrsol}
 Let $\Omega \subseteq \mathbb{R}^{d}$ be an open bounded set with $C^2$ boundary  and let $T > 0 $. Let  $b \in 
 L^{\infty}((0,T) \times \Omega; \mathbb{R}^{d}) $ be a nearly incompressible vector field
  with density $\rho $, see Definition~\ref{near-incom}. Fix $\overline{u} \in L^\infty (\Omega)$ and $\overline{g} \in L^\infty (\Gamma^-)$. 
 We say that a function $u \in L^{\infty}((0,T)\times \Omega)$ is a distributional solution of \eqref{prob-2} if  the following conditions are satisfied: 
 \begin{itemize}
 \item[i.] $u$ satisfies \eqref{iden-2};
 \item[ii.] $(\rho u)_{0}= \overline{u} \rho_0 $; 
 \item[iii.] $\emph{Tr}(\rho u b)= \overline{g} \emph{Tr}(\rho b) $ on the set $\Gamma^{-}$.
 \end{itemize} 
 \end{definition}
 \section{Proof of Theorem~\ref{IBVP-NC}: existence of solution}
 \label{s:exi}
 In this section we establish the existence part of Theorem~\ref{IBVP-NC}, namely we prove the  existence of functions $u \in L^{\infty}((0,T)\times \Omega) $ and $w \in L^{\infty}(\Gamma^{0}\cup \Gamma^+ ) $ such that for every $\psi \in C^{\infty}_{c}([0,T)\times \mathbb{R}^{d})$,
 \begin{equation}
 \int_{0}^{T} \int_{\Omega} \rho u (\partial_{t} \psi+b \cdot \nabla \psi) \ dx dt= \int_{\Gamma^{-}} \overline{g} \text{Tr}(\rho b) \psi \ d\mathcal{H}^{d-1} dt +\int_{\Gamma^{+}\cup \Gamma^0}  \! \!\text{Tr}(\rho b) \psi w \ d\mathcal{H}^{d-1} dt-\int_{\Omega} \rho_0
 \ \overline{u}\ \psi(0,\cdot)\ dx .
 \label{weak-exist1}
 \end{equation}
We proceed as follows: first, in~\S~\ref{ss:as} we introduce an approximation scheme. Next, in~\S~\ref{ss:limit} we pass to the limit and establish existence.
 \subsection{Approximation scheme}
 \label{ss:as}
In this section we rely on the analysis in~\cite[\S~3.3]{Delellis1}, but we employ a more refined approximation scheme which guarantees strong convergence of the traces. 

We set $\Lambda: =(0, T) \times \Omega$ and we recall that by assumption $\rho \in BV(\Lambda) \cap L^\infty (\Lambda).$ We apply Theorem~\ref{t:traceafp} and we select a sequence $\{ \tilde \rho_m \} 
\subseteq C^\infty (\Lambda)$ satisfying~\eqref{e:tracefp} and~\eqref{e:four}. Next, we set 
\begin{equation}
\label{e:rhoenne}
        \rho_m: = \frac{1}{m} + \tilde \rho_m \ge \frac{1}{m}.  
\end{equation}
We then apply Theorem~\ref{t:traceafp}  to the function $b \rho$ and we set 
\begin{equation}
\label{e:benne}
        b_m : = \frac{\widetilde{(b \rho)}_m}{\rho_m}. 
\end{equation}
Owing to Theorem~\ref{t:traceafp} we have 
\begin{equation}
\label{e:elle1conv}
      \rho_m \to \rho \; 
      \text{strongly in $L^1 ((0, T) \times \Omega)$}, \quad 
      b_m \rho_m \to b \rho
      \; 
      \text{strongly in $L^1 ((0, T) \times \Omega;\R^d)$}. 
\end{equation}
and, by using  the identity~\eqref{e:equal}, 
\begin{equation}
\label{e:traceconv}
\begin{split}
         \mathrm{Tr} (\rho_m) \to \mathrm{Tr} (\rho)& \; \text{strongly in $L^1 (\Gamma)$},  \quad 
         \mathrm{Tr} (\rho_mb_m) \to \mathrm{Tr} (\rho b) \; 
         \text{strongly in $L^1 (\Gamma)$}, \\
         & \quad 
         \rho_{m0} \to \rho_0  \; 
         \text{strongly in $L^1 (\Omega)$}. 
         \end{split}
\end{equation}
Note, furthermore, that 
\begin{equation}
\label{e:linftytraces}
     \| \mathrm{Tr} (b_m \rho_m ) \|_{L^\infty} \stackrel{\eqref{e:maxprintraces1}}{\leq}       \| b_m \rho_m  \|_{L^\infty}
     \stackrel{\eqref{e:four}}{\leq} 4 \| b \rho   \|_{L^\infty}. 
\end{equation}
In the following, we will use the notation
\begin{equation}
\label{e:gammadef}
  \Gamma_m^- : = \big\{(t, x) \in \Gamma: \; \mathrm{Tr} (\rho_m b_m)
  < 0 \big\}, 
  \qquad  
  \Gamma_m^+ : = \big\{(t, x) \in \Gamma: \; \mathrm{Tr} (\rho_m b_m) > 0 \big\}
\end{equation}
Finally, we extend the function $\overline{g}$ to the whole $\Gamma$ by setting it equal to $0$ outside $\Gamma^-$ and we construct two sequences $\{ \overline{g}_m \} \subseteq C^1 (\Gamma)$ and 
$\{\overline{u}_m \} \subseteq C^\infty (\Omega)$ such that 
\begin{equation}
\label{e:convbdata}
      \overline{g}_m \to \overline{g} \; \text{strongly in $L^1 (\Gamma)$}, \quad 
      \overline{u}_m \to \overline{u} \; \text{strongly in $L^1 (\Omega)$}
\end{equation} 
and 
\begin{equation}
\label{e:tomaxprin}
    \| \overline{g}_m \|_{L^\infty} \leq \| \overline{g} \|_{L^\infty}, \quad 
    \| \overline{u}_m \|_{L^\infty} \leq \| \overline{u} \|_{L^\infty}. 
\end{equation}
We can now define the function $u_m$ as the solution of the initial-boundary value problem  
 \begin{equation}
 \left\{
 \begin{array}{ll}
 \partial_{t} u_m+b_m \cdot \nabla u_m=0 & \text{on $(0, T) \times \Omega$} \\
 u_m=\overline{u}_m & \text{at $t=0$}\\ 
 u_m= \overline{g}_m  & \text{on} \; \tilde \Gamma^{-}_m,
 \end{array}
 \right.
 \label{exist3}
 \end{equation}
where $\tilde \Gamma^-_m$ is the subset of $\Gamma$ such that the characteristic lines of $b_m$ starting at a point in $\tilde \Gamma^-_m$ are entering $(0, T) \times \Omega$. We recall~\eqref{e:gammadef} and we point out that  
$$ 
   \Gamma^-_m 
  \subseteq
  \tilde \Gamma^-_m \subseteq 
    \big\{ (t, x) \in \Gamma: 
  \; b_m \cdot \vec n \leq 0 \big\}. 
$$
In the previous expression, $\vec n$ denotes as the outward pointing, unit normal vector to $\partial \Omega$. By using the classical method of characteristics (see also~\cite{Bar}) we establish the existence of a solution $u_m$ satisfying 
 \begin{equation}
 \Vert u_m\Vert_{\infty} \leq \max\{\Vert \overline{u}_m \Vert_{\infty}, \Vert \overline{g}_m \Vert_{\infty} \} 
 \stackrel{\eqref{e:tomaxprin}}{\leq} \max \{\Vert \overline{u} \Vert_{\infty}, \Vert \overline{g} \Vert_{\infty} \}.
 \label{mp}
 \end{equation}
We now introduce the function $h_m$ by setting   
  \begin{equation}
  \label{e:accaenne}
      h_m : = \partial_t \rho_m + \mathrm{div} (b_m \rho_m)
  \end{equation}
and by using  the equation at the first line of~\eqref{exist3} we get that 
$$
    \partial_t (\rho_m u_m ) + \mathrm{div} (b_m \rho_m u_m ) = h_m u_m. 
$$
Owing to the Gauss-Green formula, this implies that,
for every $\psi \in C^\infty_c ([0, T) \times \R^d)$, 
\begin{equation}
 \begin{aligned}
 &\int_{0}^{T} \int_{\Omega} \rho_m u_m [\partial_{t} \psi+ b_m \cdot \nabla \psi ] \ dx dt
 + \int_0^T \int_\Omega h_m u_m \psi \, dx dt 
  \\  
 &\quad = -\int_{\Omega} \psi(0,x) \overline{\rho}_{m0} \overline{u}_{m}  \ dx- \int_{0}^{T} \! \! \int_{\partial \Omega} 
\psi u_m \rho_m b_m \cdot \vec n \, d\mathcal{H}^{d-1} dt \\
 & \quad = 
  -\int_{\Omega} \psi(0,x) \overline{\rho}_{m0} \overline{u}_{m}  \ dx- 
  \int_{0}^{T} \! \! \int_{\partial \Omega} \mathbf{1}_{\Gamma_m^-}
 \overline{g}_{m} \psi \mathrm{Tr} (\rho_m b_m ) d\mathcal{H}^{d-1} dt 
 - \int_{0}^{T} \! \! \int_{\partial \Omega} \mathbf{1}_{\Gamma_m^+}
u_m \psi \mathrm{Tr} (\rho_m b_m ) d\mathcal{H}^{d-1} dt.
 \end{aligned}
 \label{weak-exist2}
 \end{equation}
In the above expression, we have used the notation introduced in~\eqref{e:gammadef} and the fact that 
$\mathrm{Tr} (\rho_m b_m )=0$ on~${\Gamma \setminus (\Gamma^-_m \cup \Gamma^+_m)}$.  
 
\subsection{Passage to the limit}
\label{ss:limit}
Owing to the uniform bound~\eqref{mp}, there are a subsequence of $\{ u_m \}$ (which, to simplify notation, we do not relabel) and a function $u \in L^\infty ((0, T) \times \Omega$
such that
 \begin{equation}
 \label{e:uweaks}
       u_m \weaks u \; \text{weakly$^\ast$ in $L^\infty ((0, T) \times \Omega)$. }
 \end{equation}
The goal of this paragraph is to show that the function $u$ in~\eqref{e:uweaks} is a distributional solution of~\eqref{prob-2} by passing to the limit in~\eqref{weak-exist2}. 
We first introduce a technical lemma. 
\begin{lemma}
\label{l:meyerserrin}
We can construct the approximating sequences $\{ \rho_m \}$ and $\{ b_m \}$ 
in such a way that the sequence $\{ h_m \}$ defined as in~\eqref{e:accaenne} satisfies 
\begin{equation}
\label{e:convaccaemme}
     h_m \to 0 \; \text{strongly in $L^1 ((0, T) \times \Omega)$}. 
\end{equation} 
\end{lemma} 
The proof of Lemma~\ref{l:meyerserrin}  is deferred to~\S~\ref{s:proof1} . For future reference, we state the next simple convergence result as a lemma. 
\begin{lemma}
\label{l:traces} 
Assume that 
\begin{equation}
\label{e:hyp}
       \mathrm{Tr} (\rho_m b_m )\to \mathrm{Tr}(\rho b)
       \; \text{strongly in $L^1 (\Gamma)$}.
\end{equation}
Let $\Gamma^-_m$ and $\Gamma^+_m$ as in~\eqref{e:gammadef} and $\Gamma^-$ and $\Gamma^+$ as in~\eqref{e:gamma}, respectively. Then, up to subsequences,  
\begin{equation}
\label{e:convchar1}
   \mathbf{1}_{\Gamma^-_m} \to \mathbf{1}_{\Gamma^-} +
   \mathbf{1}_{\Gamma'} \; \text{strongly in $L^1 (\Gamma)$}
\end{equation}
and 
\begin{equation}
\label{e:convchar2}
   \mathbf{1}_{\Gamma^+_m} \to \mathbf{1}_{\Gamma^+} +
   \mathbf{1}_{\Gamma''} \; \text{strongly in $L^1 (\Gamma)$},
\end{equation}
where $\Gamma'$ and $\Gamma''$ are (possibly empty) measurable sets satisfying
\begin{equation}
\label{e:subsetgamma0}
  \Gamma', \Gamma'' \subseteq \Gamma^0. 
\end{equation}
\end{lemma}
\begin{proof}[Proof of Lemma~\ref{l:traces}]
Owing to~\eqref{e:hyp} we have that, up to subsequences, the sequence 
$\{ \mathrm{Tr} (\rho_m b_m) \}$ satisfies 
$$
   \mathrm{Tr} (\rho_m b_m) (t, x) \to \mathrm{Tr} (\rho b)(t, x), 
   \quad \text{for $\mathcal{L}^1 \otimes \mathcal{H}^{d-1}$-almost 
   every $(t, x) \in \Gamma.$}
$$
Owing to the Lebesgue Dominated Convergence Theorem, this implies~\eqref{e:convchar1} and~\eqref{e:convchar2}. 
\end{proof}
We can now pass to the limit in all the terms in~\eqref{weak-exist2}. First, by combining~\eqref{e:elle1conv},~\eqref{mp},~\eqref{e:uweaks} and~\eqref{e:convaccaemme} we get that 
\begin{equation}
\label{e:conv11}
  \int_{0}^{T} \! \! \int_{\Omega} \rho_m u_m [\partial_{t} \psi+ b_m \cdot \nabla \psi ] \ dx dt
 + \int_0^T \! \! \int_\Omega h_m u_m \psi \, dx dt 
 \to \int_{0}^{T} \! \! \int_{\Omega} 
 \rho u [\partial_{t} \psi+ b \cdot \nabla \psi ] \ dx dt, 
\end{equation}
 for every $ \psi \in C^\infty_c ([0, T) \times \R^d)$. Also, by combining the second line of~\eqref{e:traceconv} with~\eqref{e:convbdata} and~\eqref{e:tomaxprin} we arrive at 
\begin{equation}
\label{e:conv21}
  \int_{\Omega} \psi(0,x) {\rho}_{m0} \overline{u}_{m}  \ dx
  \to 
  \int_{\Omega} \psi(0,x) {\rho}_{0} \overline{u}  \ dx , 
\end{equation}
for every $ \psi \in C^\infty_c ([0, T) \times \R^d). $ Next, we combine~\eqref{e:traceconv},~\eqref{e:convbdata},~\eqref{e:tomaxprin},~\eqref{e:convchar1},~\eqref{e:subsetgamma0} and the fact that 
$\mathrm{Tr}(\rho b) =0 $ on $\Gamma^0$ to get that 
\begin{equation}
\label{e:conv4}
\begin{split}
\int_{0}^{T} \! \! \int_{\partial \Omega} \mathbf{1}_{\Gamma_m^-}
 \overline{g}_{m} \psi \mathrm{Tr} (\rho_m b_m ) d\mathcal{H}^{d-1} dt \to
 & 
 \int_{0}^{T} \! \! \int_{\partial \Omega} \mathbf{1}_{\Gamma^-}
 \overline{g} \psi \mathrm{Tr} (\rho b ) d\mathcal{H}^{d-1} dt \\
 & = 
 \int_{0}^{T} \! \! \int_{\Gamma^-} 
 \overline{g} \psi \mathrm{Tr} (\rho b ) d\mathcal{H}^{d-1} dt,
 \end{split}
\end{equation}
for every $\psi \in C^\infty_c ([0, T) \times \Omega; \R^d)$. We are left with the last term in~\eqref{weak-exist2}: first, we denote by $u_{m{|_\Gamma}}$ the restriction of $u_m$ to $\Gamma$. Since $u_m$ is a smooth function, then 
$$
   \| u_{m{|_\Gamma}} \|_{L^\infty (\Gamma)} \leq 
   \| u_m \|_{L^\infty ((0, T) \times \Omega)}
   \stackrel{\eqref{mp}}{\leq}  
   \max \big\{ \| \bar u \|_{L^\infty}, \| \bar g \|_{L^\infty} \big\}
$$
and hence there is a function $w \in L^\infty (\Gamma)$ such that, up to subsequences, 
\begin{equation}
\label{e:convw}
    u_{m{|_\Gamma}} \weaks w \; \text{weakly$^\ast$ in $L^\infty (\Gamma)$}. 
\end{equation}
By combining~\eqref{e:traceconv},~\eqref{e:convchar2},~\eqref{e:convw} and the fact that $\mathrm{Tr} (\rho b) =0$ on $\Gamma^0$ we get that 
\begin{equation}
\begin{split}
\label{e:conv5} 
        \int_{0}^{T} \! \! \int_{\partial \Omega} \mathbf{1}_{\Gamma_m^+}
u_m \psi \mathrm{Tr} (\rho_m b_m ) d\mathcal{H}^{d-1} dt \to &
       \int_{0}^{T} \! \! \int_{\partial \Omega} \mathbf{1}_{\Gamma^+}
w \psi \mathrm{Tr} (\rho b ) d\mathcal{H}^{d-1} dt \\
   & = \! \! \int_{\Gamma^+ \cup \Gamma^0}
w \psi \mathrm{Tr} (\rho b ) d\mathcal{H}^{d-1} dt . 
\end{split}
\end{equation}
By combining~\eqref{e:conv11},~\eqref{e:conv21},~\eqref{e:conv4} and~\eqref{e:conv5} we get that $u$ satisfies~\eqref{weak-exist1} and this establishes existence of a distributional solution of~\eqref{prob-2}.  
\subsection{Proof of Lemma~\ref{l:meyerserrin}}
\label{s:proof1}
To ensure that~\eqref{e:convaccaemme} holds we use the same approximation \emph{\`a la} Meyers-Serrin as in~\cite[pp.122-123]{AFP}. We now recall some details of the construction. First, we fix a countable family of open sets 
$\big\{ \Lambda_h \big\}$ such that 
\begin{itemize}
\item[i.] $\Lambda_h$ is compactly contained in $\Lambda$, for every $h$; 
\item[ii.]  $\big\{ \Lambda_h \big\}$ is a covering of $\Lambda$, namely
$$
    \bigcup_{h=1}^\infty \Lambda_h = \Lambda;
$$
\item[iii.] every point in $\Lambda$ is contained in at most $4$ sets $\Lambda_h$. 
\end{itemize}
Next, we consider a partition of unity associated to $\big\{ \Lambda_h \big\}$, namely a countably family of smooth, nonnegative functions $\{ \zeta_h \}$ such that 
\begin{itemize}
\item[iv.] we have 
\begin{equation}
\label{e:isone}
\sum_{h=1}^\infty \zeta_h \equiv1
\quad \text{in $\Omega$}
;
\end{equation}
\item[v.] for every $h>0$, the support of $\zeta_h$ is contained in $\Lambda_h$.
\end{itemize}
Finally, we fix a convolution kernel $\eta: \R^{d+1} \to \R^+$ and we define $\eta_\ee$ by setting 
$$
    \eta_\ee (z) : = \frac{1}{\ee^{d+1}} \eta 
    \left(
    \frac{z}{\ee}
    \right) 
$$    
For every $m>0$ and $h>0$ we can choose $\ee_{mh}$ in such a way that 
$(\rho \zeta_h) \ast \eta_{\ee_{mh}} $ is supported in $\Lambda_h$ and furthermore 
\begin{equation}
\label{e:ms2}
   \int_0^T \! \! \int_\Omega 
   | \rho \zeta_h - ( \rho \zeta_h) \ast \eta_{\ee_{mh}}   |+
   | \rho \, \partial_t \zeta_h - ( \rho  \, \partial_t \zeta_h) \ast \eta_{\ee_{mh}}|  
   +
   | \rho b \cdot \nabla \zeta_h - ( \rho b \cdot \nabla \zeta_h) \ast \eta_{\ee_{mh}}| dx dt 
   \leq \frac{2^{-h}}{m}. 
\end{equation}    
We then define $\tilde \rho_m$ by setting 
\begin{equation}
\label{e:ms3}
   \tilde \rho_m : = \sum_{h=1}^\infty
   (\rho \zeta_h) \ast \eta_{\ee_{mh}} .
\end{equation}
The function $(\widetilde{\rho b})_m$ is defined analogously. Next, we proceed as in~\cite[p.123]{AFP} and we point out that 
\begin{equation*}
\begin{split}
    h_m \stackrel{\eqref{e:accaenne}}{=} &
    \partial_t \rho_m + \mathrm{div} ({\rho_m b_m}) 
    \stackrel{\eqref{e:accaenne}}{=}
   \underbrace{\sum_{h=1}^\infty
   (\partial_t \rho \zeta_h) \ast \eta_{\ee_{mh}} +
    \sum_{h=1}^\infty
   (\mathrm{div} (\rho b)  \zeta_h) \ast \eta_{\ee_{mh}}}_{= 0
   \; \text{by~\eqref{e:continuityrho} }  } 
    \\ &\quad + 
    \sum_{h=1}^\infty
   (\rho \, \partial_t \zeta_h) \ast \eta_{\ee_{mh}} +
    \sum_{h=1}^\infty
   (\rho b \cdot \nabla \zeta_h) \ast \eta_{\ee_{mh}}
   \\ & \stackrel{\eqref{e:isone}}{=}
    \sum_{h=1}^\infty
   (\rho \, \partial_t \zeta_h) \ast \eta_{\ee_{mh}} - 
   \rho \sum_{h=1}^\infty \partial_t \zeta_h  
   \quad + 
   \sum_{h=1}^\infty
   (\rho b \cdot \nabla \zeta_h) \ast \eta_{\ee_{mh}}-
   \rho b \cdot \sum_{h=1}^\infty
   \nabla \zeta_h 
\end{split}
\end{equation*}
By using~\eqref{e:ms2} we then get that 
$$
   \int_0^T \! \! \int_\Omega |h_m| dx dt \leq  \sum_{h=1}^\infty
   \frac{2^{-h}}{m} = \frac{1}{m}    
$$
and this establishes~\eqref{e:convaccaemme}.  
\label{s:proof2}
 \section{Proof Theorem~\ref{IBVP-NC}: comparison principle and uniqueness}
 \label{s:uni}
In this section we complete the proof of Theorem~\ref{IBVP-NC}. More precisely, we establish the following comparison principle.
\begin{lemma}
\label{l:uni}
         Let $\Omega$, $b$ and $\rho$ as in the statement of 
         Theorem~\ref{IBVP-NC}.  Assume $u_1$ and $u_2 \in
          L^{\infty}((0,T) \times \Omega)$ are distributional 
          solutions (in the sense of Definition~\ref{d:distrsol}) of the initial-boundary value problem~\eqref{prob-2} 
          corresponding to the initial and boundary data 
          $\overline{u}_{1} \in L^{\infty}(\Omega)$, $\overline{g}_1 \in L^\infty(\Gamma^-)$ and 
          $\overline{u}_2 \in L^{\infty}(\Omega)$, $\overline{g}_2 \in L^\infty(\Gamma^-)$, respectively. If $\overline{u}_1 \ge \overline{u}_2$ and $\overline{g}_1 \ge \overline{g}_2$, then 
          \begin{equation}
          \label{e:compa}
               \rho u_1 \ge \rho u_2 \quad a.e. \; \text{in} \; (0, T) \times \Omega. 
          \end{equation} 
\end{lemma}
Note that the uniqueness of $\rho u$, where $u$ is a distributional solution of the initial-boundary value problem~\eqref{prob-2}, immediately follows from the above result.
\begin{proof} [Proof of Lemma~\ref{l:uni}]
Let us define the function 
$$
\tilde{\beta}(u)= 
\left\{
\begin{array}{ll}
u^2 & u \geq 0 \\
0 & u<0.
\end{array}\right.
$$
In what follows, we shall prove that the identity $\rho\ \tilde{\beta}(u_{2}-u_{1})=0 $ holds almost everywhere, whence the comparison principle follows. To see this, we proceed as described below.
First, we point out that, since the equation at the 
first line of~\eqref{prob-2} is linear, then $u_2-u_1$ is a distributional solution of the initial boundary value problem with data $\overline{u}_2 - \overline{u}_1$, $\overline{g}_2 - \overline{g}_1$.  In particular, for every $ \psi \in C^{\infty}_{c}([0,T) \times \mathbb{R}^{d} )$ we have 
\begin{equation}
\int_{0}^{T} \int_{\Omega} \rho (u_{2}-u_{1}) (\partial_{t} \psi +b \cdot \nabla \psi) \ dx dt= \int_{0}^{T} \int_{\partial \Omega} [\text{Tr}(\rho u_2 b) - \text{Tr}(\rho u_{1} b)] \ \psi \ d\mathcal{H}^{d-1} dt -\int_{\Omega} \psi(0,\cdot)  {\rho}_0 (\overline{u}_{2}-\overline{u}_{1}) \ dx
\label{e7}
\end{equation}
and 
\begin{equation}
\label{e:ntraces}
    \text{Tr}(\rho u_2 b)  = \overline{g}_2 \text{Tr}(\rho b), \quad 
     \text{Tr}(\rho u_1 b)  = \overline{g}_1 \text{Tr}(\rho b)
     \quad \text{on $\Gamma^-$}.  
\end{equation}
Note that~\eqref{e7} implies that 
\begin{equation}
\int_{0}^{T} \int_{\Omega} \rho (u_{2}-u_{1}) (\partial_{t} \phi+b \cdot \nabla \phi) \ dx dt=0, \quad \forall \phi \in C^{\infty}_{c}((0,T) \times \Omega).
\label{e5}
\end{equation}
By using~\cite[Lemma 5.10]{Delellis1} (renormalization property inside the domain), we get 
\begin{equation}
\int_{0}^{T} \int_{\Omega} \rho \ \tilde{\beta}(u_{2}-u_{1})(\partial_{t} \phi+b \cdot \nabla \phi) \ dx dt=0, \qquad \forall \phi \in C^{\infty}_{c}((0,T) \times \Omega).
\label{e10}
\end{equation}
We next apply Lemma \ref{trace-existence} to the function $\tilde{\beta}(u_{2}-u_{1})$ to infer that there are bounded functions $\text{Tr}(\rho \tilde{\beta}(u_{2}-u_{1}) b)$ and $(\rho \tilde{\beta}(u_{2}-u_{1}))_{0} $ such that, for every 
$ \psi \in C^{\infty}_{c}([0,T) \times \mathbb{R}^{d} ),$ we have
\begin{equation}
\int_{0}^{T} \int_{\Omega} \rho \ \tilde{\beta}(u_{2}-u_{1}) (\partial_{t} \psi +b \cdot \nabla \psi) \ dx dt= \int_{0}^{T} \int_{\partial \Omega} \text{Tr}(\rho \ \tilde{\beta}(u_{2}-u_{1}) b) \ \psi \ d\mathcal{H}^{d-1} dt -\int_{\Omega} \psi(0,\cdot)  (\rho \ \tilde{\beta}(u_{2}-u_{1}))_{0} \ dx.
\label{e11}
\end{equation}
We recall~\eqref{e7} and we apply Lemma \ref{trace-renorm} (trace renormalization property) with $w= u_2 -u_1$, $h= \tilde \beta$, $B=(\rho,\rho b) $, $\Lambda = \mathbb R^{d+1}$ and $\Lambda'=(0,T)\times \Omega$. We recall that the vector field $P$ is defined by setting $P:= (\rho, \rho b)$ and we get 
\begin{equation}
\begin{aligned}
(\rho\ \tilde{\beta} (u_{2}-u_{1}))_{0}=- 
\text{Tr}(\tilde{\beta}(u_{2}-u_{1}) P,\partial \Lambda')
\Big\vert_{\{0\} \times \Omega}&= - \tilde{\beta}\left(\frac{(\rho (u_{2}-u_{1}))_{0}}{\text{Tr}(P,\partial \Lambda')\Big\vert_{\{0\}\times \Omega}} \right)
\text{Tr}(P,\partial \Lambda')\Big\vert_{\{ 0\} \times \Omega}\\
&=-\tilde{\beta}\left( \frac{\rho_0 (\overline{u}_{2}-\overline{u}_{1})}{\overline{\rho}} \right) \rho_0 \\
& =0, \; \text{since} \ \overline{u}_{1} \geq \overline{u}_{2} \phantom{\int} 
\end{aligned}
\label{e12}
\end{equation}
and 
\begin{equation}
\begin{aligned}
\text{Tr}(\rho \ \tilde{\beta}(u_{2}-u_{1}) b) &=
\text{Tr}(\tilde{\beta}(u_{2}-u_{1}) \rho, \partial \Lambda')\Big\vert_{(0,T) \times \partial \Omega} = 
\tilde{\beta} \left(
\frac{\text{Tr}((u_{2}-u_{1})\rho, \partial \Lambda')\Big\vert_{(0,T) \times \partial \Omega}}{\text{Tr}(P, \partial \Lambda')\Big\vert_{(0,T) \times \partial \Omega}} 
\right) \text{Tr}(P, \partial \Lambda')\Big\vert_{(0,T) \times \partial \Omega}\\
&=\tilde{\beta}\left(\frac{\text{Tr}(\rho (u_{2}-u_{1}) b)}{\text{Tr}(\rho b)} \right) \text{Tr}(\rho b).  
\end{aligned}
\notag
\end{equation}
By recalling~\eqref{e:ntraces} and the inequality $\bar g_1 \ge \bar g_2$, we conclude that
$$
 \text{Tr}(\rho \ \tilde{\beta}(u_{2}-u_{1}) b)  = 0 \quad \text{on $\Gamma^-$}
$$
and, since $\tilde \beta \ge 0$, that 
\begin{equation}
\label{e:ntracein}
 \text{Tr}(\rho \ \tilde{\beta}(u_{2}-u_{1}) b)  \ge 0
  \quad \text{on $\Gamma$.}
\end{equation}
We now choose a test function $\nu \in C^\infty_c (\mathbb R^d)$ in such a way that $\nu \equiv 1$ on the bounded set $\Omega$. Note that 
\begin{equation}
\label{e:zerozero}
    \partial_t \nu + b \cdot \nabla \nu =0 \quad \text{on $(0, T) \times \Omega$.}
\end{equation}
Next we choose a sequence of functions $\chi_{n} \in {C}^{\infty}_{c}([0,+\infty))$ that satisfy
 \[\chi_{n} \equiv 1 \ \text{on}\ [0,\bar{t}],\ \chi_{n}\equiv 0\ \text{on}\ [\bar{t}+\frac{1}{n},+\infty),\ \chi'_{n} \leq 0, \]
 and we define 
 \[\psi_{n}(t,x):= \chi_{n}(t) \nu(x), \ (t,x)\in [0,T)\times \mathbb{R}^{d}.\] 
Note that $\psi$ is smooth, non-negative and compactly supported in 
$[0,T)\times \mathbb{R}^{d}$. By combining the identities \eqref{e11},~\eqref{e12} and the inequality~\eqref{e:ntracein}, we get 
 \begin{equation}
 \begin{aligned}
 0 &\leq \int_{0}^{T} \int_{\Omega} \rho\  \tilde{\beta}(u_{2}-u_{1}) [\partial_{t}(\chi_{n} \nu)+b \cdot \nabla (\chi_{n} \nu)] \ dx dt \\
 & = \int_{0}^{T} \int_{\Omega} \nu \rho\  \tilde{\beta}(u_{2}-u_{1}) \chi'_{n} \ dx dt+ \int_{0}^{T} \int_{\Omega} \chi_{n} \rho \ \tilde{\beta}(u_{2}-u_{1}) (\partial_{t} \nu
 +b \cdot \nabla \nu) \ dx dt \\
 & \stackrel{\eqref{e:zerozero}}{=} \int_{0}^{T} \int_{\Omega} \nu \rho \ \chi'_{n} \ \tilde{\beta}(u_{2}-u_{1})  \ dx dt. \\
 \end{aligned}
  \notag 
 \end{equation}
Passing to the limit as $n \rightarrow +\infty $ and noting that $\chi'_{n} \rightarrow -\delta_{\bar{t}} $ as $n \rightarrow \infty $ in the sense of distributions 
and recalling that $\nu \equiv 1$ on $\Omega$ we obtain
\begin{equation}
\int_{\Omega}  \rho(\bar{t},\cdot) \tilde{\beta}(u_{2}-u_{1})(\bar{t},\cdot) \leq 0.
  \notag
\end{equation}
Since the above inequality is true for arbitrary $\bar t \in [0, T]$, we can conclude that 
\begin{equation}
 \begin{aligned}
 \rho \ \tilde{\beta}(u_2-u_1)=0,\ \text{for almost every}\ (t,x)
\Rightarrow \rho u_{1} \geq \rho u_{2}, \ \text{for almost every}\ (t,x).
\end{aligned}
\label{e14}
\end{equation}
This concludes the proof of Lemma~\ref{l:uni}. 
\end{proof}

 \section{Stability and space continuity properties}
 \label{s:ssc}
 In this section, we discuss some qualitative properties of solutions of the initial-boundary value problem~\eqref{prob-2}. First, we establish Theorem~\ref{stability-weak}, which establishes (weak) stability of solutions with respect to perturbations in the vector fields and the data. Theorem~\ref{stability-strong} implies that, under stronger hypotheses, we can establish strong stability. Finally, Theorem~\ref{space-continuity} establishes space continuity properties. 
 \begin{theorem}\label{stability-weak}
Let $T>0$ and let $\Omega \subseteq \R^d$ be an open and bounded set with $C^2$ boundary.  
Assume that
$$
b_{n}, b \in BV((0,T) \times \Omega; \mathbb{R}^{d}) \cap L^{\infty}((0,T) \times \Omega; \mathbb{R}^{d}), \qquad \rho_{n},\rho \in BV((0,T) \times \Omega) \cap L^{\infty}([0,T) \times \Omega)
$$ satisfy
 \begin{equation}
 \begin{aligned}
 \partial_{t} \rho_{n}+\mathrm{div} (b_{n}\rho_{n})=0,\\
 \partial_{t} \rho+\mathrm{div} (b \rho)=0,
 \end{aligned}
 \label{stability-1} 
 \end{equation}
 in the sense of distributions on $(0, T) \times \Omega$. Assume furthermore that 
 \begin{equation}
  0 \leq \rho_{n}, \rho \leq C \; \text{and} \; \Vert b_{n} \Vert_{\infty}\ \text{is uniformly bounded},
 \label{stability-2}
 \end{equation}
 \begin{equation}
  (b_{n},\rho_{n}) \xrightarrow[n \rightarrow \infty]{} (b,\rho) \  \text{strongly in} \ L^{1}((0,T) \times \Omega; \R^{d+1}),
  \label{stability-3}
  \end{equation}
  \begin{equation}  
  \rho_{n0} \xrightarrow[n \rightarrow \infty]{} \rho_0 
  \; \text{strongly in}\ L^{1}(\Omega),
  \label{stability-4}
  \end{equation}
  \begin{equation}
   \emph{Tr}(\rho_{n} b_n) \xrightarrow[n \rightarrow \infty]{} \emph{Tr}(\rho b)\  \text{strongly in}\ L^{1}(\Gamma),
  \label{stability-5}
  \end{equation}
Let $u_{n} \in L^{\infty}((0,T) \times \Omega) $ be a distributional solution 
(in the sense of Definition~\ref{d:distrsol}) of the initial-boundary value problem  
 \begin{equation}
 \label{e:ibvpapp}
 \left\{
 \begin{array}{lll}
 \partial_{t}(\rho_{n} u_{n})+\mathrm{div}(\rho_{n} u_{n} b_{n})=0 &
  \text{in} \ (0,T)\times \Omega \\
  u_{n}=\overline{u}_{n} & \text{at $t=0$}\\ 
  u_{n} =\overline{g}_{n}  &  \text{on}
 \ \Gamma_{n}^{-} \\
 \end{array}
 \right.
 \end{equation}
 and $u \in L^{\infty}((0,T) \times \Omega) $ be a distributional solution of the equation
 \begin{equation}
 \label{e:ibvplimit}
 \left\{
 \begin{array}{lll}
 \partial_{t}(\rho u)+\mathrm{div}(\rho  u b)=0 & \text{in} \ (0,T)\times \Omega \\
  u=\overline{u} &  \text{at $t=0$}\\ 
  u=\overline{g} &  \text{on}\ \Gamma^{-}.
 \end{array}
 \right.
 \end{equation}
If  $u_m, \bar u \in L^\infty (\Omega)$ and $ \overline{g}_{n}, \bar g \in L^\infty (\Gamma)$ satisfy
\begin{equation}
 \overline{u}_{n} \stackrel{\ast}{\rightharpoonup} \overline{u}\ \text{weak-$^\ast$ in}\ L^{\infty}(\Omega),
 \label{stability-7}
 \end{equation}
 \begin{equation}
 \overline{g}_{n}   \stackrel{\ast}{\rightharpoonup}  \overline{g}   \;
\text{weak-$^\ast$ in}\ L^{\infty}(\Gamma) ,
 \label{stability-8}
 \end{equation}
then 
 \begin{equation}
 \rho_{n} u_{n} \stackrel{*}{\rightharpoonup}
  \rho u \ \text{weak-* in}\ L^{\infty}((0,T) \times \Omega)
 \label{stability-10}
 \end{equation} 
  and
  \begin{equation}
 \emph{Tr}(\rho_{n} u_{n} b_{n}) \stackrel{*}{\rightharpoonup} 
 \emph{Tr}(\rho u b) \ \text{weak-* in}\ L^{\infty}(\Gamma).
 \label{stability-9}
 \end{equation}
 \end{theorem}
 Note that in the statement of the above theorem $\overline{g}_m$ and $\overline{g}$ are functions defined on the whole $\Gamma$, although the values of $\rho_m u_m$ and $\rho u$ are only determined by their values on $\Gamma^-_m$ and $\Gamma^-$, respectively. 
 \begin{proof}
 We proceed according to the following steps. \\
 {\sc Step 1:} we apply Theorem~\ref{IBVP-NC} and we infer that the function $\rho_n u_n$ satisfying~\eqref{e:ibvpapp} is unique. Also, without loss of generality, we can redefine the function $u_n$ on the set $\{\rho_n=0\}$  in such a way that $u_n$ satisfies the maximum principle~\eqref{e:maxprin}. Owing to~\eqref{stability-9}, the sequences $\| \overline{u}_m \|_{L^\infty}$ and $\| \overline{g}_m \|_{L^\infty}$ are both uniformly bounded and by the maximum principle so is $\| u_m \|_{L^\infty}$. Also, by combining~\eqref{e:maxprintraces2} and~\eqref{stability-2} we infer that the sequence $\Vert \text{Tr}(\rho_{n} b_{n} u_{n}) \Vert_{\infty} $ is also uniformly bounded. We conclude that, up to subsequences (which we do not label to simplify the notation), we have 
 \begin{equation}
 \begin{aligned}
 & u_{n} \stackrel{*}{\rightharpoonup} r_{1} \ \text{weak-* in} \ L^{\infty}((0,T) \times \Omega),\\
 & \text{Tr}(\rho_{n}  u_{n} b_{n}) \stackrel{*}{\rightharpoonup} r_{2} \ \text{weak-* in} \ L^{\infty}(\Gamma)
 \end{aligned}
 \label{stability-11}
 \end{equation}
 for some $r_{1} \in L^{\infty}((0,T) \times \Omega)$ and $r_{2} \in L^{\infty}(\Gamma)$.
 By using \eqref{iden-2} and \eqref{iden-3}, we get that
 \begin{equation}
 \int_{0}^{T} \int_{\Omega} \rho r_{1} (\partial_{t} \phi+b \cdot \nabla \phi)\ dx dt=0, \quad  \forall \phi \in C^{\infty}_{c}((0,T) \times \Omega), 
 \label{stability-12}
 \end{equation}
 and
 \begin{equation}
 \int_{0}^{T} \int_{\Omega} \rho r_{1} (\partial_{t} \psi+b  \nabla \psi)\ dx dt= \int_{0}^{T} \int_{\partial \Omega} r_{2} \psi\ d\mathcal{H}^{d-1} dt -\int_{\Omega} \psi(0,\cdot) {\rho}_0\  \overline{u}\ dx ,\ \forall \psi \in C^{\infty}_{c}([0,T) \times \mathbb{R}^{d}).
 \label{stability-13}
 \end{equation}
 From Lemma \ref{trace-existence}, it also follows that 
 \begin{equation}
 r_{2}= \text{Tr}(\rho r_{1} b). 
 \label{stability-14}
 \end{equation}
Assume for the time being that we have established the equality
\begin{equation}
\label{e:whatww}
  r_{2}=\overline{g} \text{Tr}(\rho b), \quad \text{on}\ \Gamma^{-} ,
\end{equation}
then by recalling~\eqref{stability-14} and the uniqueness part in Theorem~\ref{IBVP-NC} we conclude that $r_1= \rho u$ and $r_2 = \text{Tr}(\rho bu) $. Owing to~\eqref{stability-11}, this concludes the proof of the theorem. \\
{\sc Step 2:} we establish~\eqref{e:whatww}.  First, we decompose $\text{Tr}(\rho_m u_m b_m) 
$ as
\begin{equation}
\label{e:decompo}
\begin{split}
   \text{Tr}(\rho_n u_n b_n)  &= \text{Tr}(\rho_n u_n b_n)
    \mathbf{1}_{\Gamma^{-}_{n}}+\text{Tr}(\rho_n u_n b_n)   \mathbf{1}_{\Gamma^{+}_{n}} + \text{Tr}(\rho_n u_n b_n)   \mathbf{1}_{\Gamma^{0}_{n}}
    \\
    & = \overline{g}_n  \text{Tr}(\rho_n  b_n)
     \mathbf{1}_{\Gamma^{-}_{n}}+\text{Tr}(\rho_n u_n b_n)   \mathbf{1}_{\Gamma^{+}_{n}}  + \text{Tr}(\rho_n u_n b_n)   \mathbf{1}_{\Gamma^{0}_{n}}  ,
     \end{split}
\end{equation}
where $\Gamma^-_n$, $\Gamma^+_n$ and $\Gamma^0_n$ are defined as in~\eqref{e:gamma}. 
By using Lemma~\ref{trace-renorm} (trace renormalization), one could actually prove that the last term in the above expression is actually $0$. This is actually not needed here. Indeed, it suffices to recall~\eqref{stability-5} and Lemma~\ref{l:traces} and point out that by combining~\eqref{e:convchar1} and~\eqref{e:convchar2} we get
\begin{equation}
\label{e:conchar3} 
   \mathbf{1}_{\Gamma^{0}_{n}} \to \mathbf{1}_{\Gamma^0} 
   - \mathbf{1}_{\Gamma'} - \mathbf{1}_{\Gamma''}. 
\end{equation}
Next, we recall that 
 the sequence $\| \text{Tr}(\rho_n u_n b_n)\|_{L^\infty}$ is uniformly bounded owing to the uniform bounds on $\| \rho_n \|_{L^\infty}$ and $\| u_n \|_{L^\infty}$. By recalling~\eqref{stability-8}, we conclude that 
 \begin{equation}
 \label{e:conv1}
 \overline{g}_n  \text{Tr}(\rho_n  b_n)
     \mathbf{1}_{\Gamma^{-}_{n}}  \stackrel{*}{\rightharpoonup} 
     \overline{g}  \, \text{Tr}(\rho  b)
     \Big( \mathbf{1}_{\Gamma^{-}} + \mathbf{1}_{\Gamma'} \Big)
     \qquad \text{weak-* in $L^\infty (\Gamma)$.}
 \end{equation}
 By recalling that $\Gamma' \subseteq \Gamma^0$ we get that $\text{Tr}(\rho  b)
 \mathbf{1}_{\Gamma'} =0$. We now pass to the weak star limit in~\eqref{e:decompo} and using~\eqref{e:convchar1},~\eqref{e:convchar2},~\eqref{stability-11},~\eqref{stability-8} and~\eqref{e:conv1} we get 
\begin{equation}
\label{e:conv2}
  r_2 =  \overline{g} \text{Tr}(\rho  b) 
   \mathbf{1}_{\Gamma^{-}}  + r_2 
    \Big( \mathbf{1}_{\Gamma^{+}} + \mathbf{1}_{\Gamma'} \Big)+ r_2 
     \Big( \mathbf{1}_{\Gamma^{0}} - \mathbf{1}_{\Gamma'} -\mathbf{1}_{\Gamma''}  \Big),
\end{equation}
which owing to the properties 
$$
   \Gamma^- \cap \Gamma^{0}= \emptyset, 
   \quad \Gamma^- \cap \Gamma'=\emptyset, \quad 
   \Gamma^- \cap \Gamma''= \emptyset 
 $$  
implies~\eqref{e:whatww}. This concludes the proof Theorem~\ref{stability-weak}. 
\end{proof} 
 \begin{theorem}\label{stability-strong}
 Under the same assumptions as in Theorem~\ref{stability-weak}, if we furthermore assume that
 \begin{equation}
 \overline{u}_{n} \xrightarrow[n \rightarrow \infty]{} \overline{u}\ \text{strongly in}\ L^{1}(\Omega),
 \label{stability-35}
 \end{equation}
 \begin{equation}
 \overline{g}_{n} \xrightarrow[n \rightarrow \infty]{} \overline{g} \ \text{strongly in}\ L^{1}(\Gamma) ,
 \label{stability-36}
 \end{equation}
 then we get 
 \begin{equation}
 \begin{aligned}
 &\rho_{n} u_{n} \xrightarrow[n \rightarrow \infty]{} \rho u \ \text{strongly in}\ L^{1}((0,T) \times \Omega),\\
 &\emph{Tr}(\rho_{n} u_{n} b_{n}) \xrightarrow[n \rightarrow \infty]{} \emph{Tr}(\rho u b) \ \text{strongly in}\ L^{1}(\Gamma).
 \end{aligned}  
 \label{stability-37}
 \end{equation}
 \end{theorem}
 \begin{proof}
First, we point out that the first equation 
 in~\eqref{stability-9} implies that 
\begin{equation}
 \label{e:ell2}
  \rho_n u_m
     {\rightharpoonup} \rho {u}\ \text{weakly in}\ L^{2}((0, T) \times \Omega ). 
\end{equation} 
Next, by using Lemma \ref{trace-renorm} (trace-renormalization property), we get that $\rho_m u^{2}_{n}$ and $\rho u^{2}$ satisfy (in the sense of distributions) 
 \begin{equation}
 \left\{
 \begin{array}{lll}
 \partial_{t}(\rho_{n} u_{n}^{2})+\text{div}(\rho_{n} u_{n}^{2} b_{n})=0 &
  \text{in} \ (0,T)\times \Omega \\
  u_{n}^{2}=\overline{u}_{n}^{2} & \text{at $t=0$}\\ 
 u^{2}_{n} =\overline{g}^{2}_{n}  &  \text{on}\ \Gamma_{n}^{-}, \\
 \end{array}
 \right.
 \notag
 \end{equation}
 and
 \begin{equation}
 \left\{
 \begin{array}{lll}
 \partial_{t}(\rho u^{2})+\text{div}(\rho u^{2} b)=0 &
  \text{in} \ (0,T)\times \Omega \\
  u^{2}=\overline{u}^{2} & \text{at $t=0$} \\ 
  u^{2} =\overline{g}^{2} &  \text{on}\ \Gamma^{-}, \\
 \end{array}
 \right.
 \notag
 \end{equation}
 respectively. Also, by combinig~\eqref{stability-7},\eqref{stability-8}, \eqref{stability-35} and \eqref{stability-36}, we get that 
 \begin{equation}
 \overline{u}^2_{n} \stackrel{*}{\rightharpoonup} \overline{u}^2\ \text{weak-$^\ast$ in}\ L^{\infty}(\Omega), \qquad 
 \overline{g}^2_{n}   \stackrel{*}{\rightharpoonup}  \overline{g}^2   \;
\text{weak-$^\ast$ in}\ L^{\infty}(\Gamma)
\notag
 \end{equation}
 and by applying Theorem~\ref{stability-weak} to $\rho_m u_m^2$ we conclude that 
 $$
     \rho_m u_m^2
     \stackrel{*}{\rightharpoonup} \rho {u}^2\ \text{weak-$^\ast$ in}\ L^{\infty}((0, T) \times \Omega ) 
 $$ 
 and that 
 \begin{equation}
 \label{e:convbur}
        \text{Tr}(\rho_{n}  u^2_{n} b_{n})
         \stackrel{*}{\rightharpoonup}
         \text{Tr}(\rho  u^2 b)
        \ \text{weak-$^\ast$ in}\ L^{\infty}(\Gamma ). 
 \end{equation}
 Since the sequence $\| \rho_m \|_{L^\infty}$ is uniformly bounded, then 
by recalling~\eqref{stability-3} we get 
$$
     \rho^2_m u_m^2
     \stackrel{*}{\rightharpoonup} \rho^2 {u}^2\ \text{weak-$^\ast$ in}\ L^{\infty}((0, T) \times \Omega )
 $$ 
 and hence 
 \begin{equation}
 \label{e:square}
     \rho^2_m u_m^2
     {\rightharpoonup} \rho^2 {u}^2\ \text{weakly in}\ L^2((0, T) \times \Omega ).
 \end{equation}
 By combining~\eqref{e:ell2} and~\eqref{e:square} we get that 
 $\rho^2_m u_m^2 \longrightarrow \rho u$ strongly in
 $L^2((0, T) \times \Omega )$ and this implies the first convergence in~\eqref{stability-37}.
 
 Next, we establish the second convergence in $L^2((0, T) \times \Omega )$. Since $\Gamma $ is a set of finite measure, from \eqref{stability-9} and \eqref{e:convbur} we can infer that
 \begin{equation}
 \begin{aligned}
 & \text{Tr}(\rho_{n}  u_{n} b_{n}) \rightharpoonup \text{Tr}(\rho  u b) \ \text{weakly in} \ L^{2}(\Gamma),\\
 &  \text{Tr}(\rho_{n}  u^{2}_{n} b_{n}) \rightharpoonup \text{Tr}(\rho  u^{2} b) \ \text{weakly in} \ L^{2}(\Gamma).
 \end{aligned}
 \label{stability-39}
 \end{equation}
 By using the uniform bounds for $\Vert \text{Tr}(\rho_{n} b_{n})\Vert_{\infty} $, we infer from the $L^{1}$ convergence of $\text{Tr}(\rho_{n} b_{n}) $ to $\text{Tr}(\rho b)$ that
 \begin{equation}
 \text{Tr}(\rho_{n} b_{n}) \xrightarrow[n \rightarrow \infty]{} \text{Tr}(\rho b) \ \text{strongly in}\ L^{2}(\Gamma).
 \label{stability-40}
 \end{equation}
Next, we apply Lemma \ref{trace-renorm} (trace renormalization property) and we get that
 \begin{equation}
 [\text{Tr}(\rho_{n} u_{n} b_{n})]^{2}= \left[\frac{\text{Tr}(\rho_{n} u_{n} b_{n})}{\text{Tr}(\rho_{n} b_{n})} \right]^{2} [\text{Tr}(\rho_{n} b_{n})]^{2}= \text{Tr}(\rho_{n} u_{n}^{2} b_{n}) \text{Tr}(\rho_{n}b_{n})
 \notag
 \end{equation}
 and
 \begin{equation}
 [\text{Tr}(\rho u b)]^{2}= \left[\frac{\text{Tr}(\rho u b)}{\text{Tr}(\rho b)} \right]^{2} [\text{Tr}(\rho b)]^{2}= \text{Tr}(\rho u^{2} b) \text{Tr}(\rho b).
 \notag
 \end{equation}
 From \eqref{stability-39} and \eqref{stability-40}, we can then conclude that
 \begin{equation}
 [\text{Tr}(\rho_{n} u_{n} b_{n})]^{2} \rightharpoonup [\text{Tr}(\rho u b)]^{2} \ \text{weakly in}\ L^{2}(\Gamma),
 \label{stability-41}
 \end{equation}
 and by recalling~\eqref{stability-39} the second convergence in \eqref{stability-37} follows.          
 \end{proof}
Finally, we establish space-continuity properties of the vector field $(\rho u, \rho u b)$
similar to those established in~\cite{Boyer,CDS1}.
 \begin{theorem}\label{space-continuity}
Under the same assumptions as in Theorem~\ref{IBVP-NC}, let $P$ be the vector field $P : = ( \rho, \rho b)$, $u$ be a distributional solution of~\eqref{prob-2} and
 $\{\Sigma_{r} \}_{r \in I} \subseteq \mathbb R^d$ be a family of graphs as in  Definition \ref{graph}. 
Also, fix $r_{0} \in I $ and let $\gamma_{0}, \gamma_{r}: (0,T) \times D \rightarrow \mathbb{R} $ be defined by 
 \begin{equation}
 \begin{aligned}
 \gamma_{0}(t,x_{1},\cdots,x_{d-1})&:= \emph{Tr}^{-}(uP,(0,T)\times \Sigma_{r_{0}})(t,x_{1},\cdots,x_{d-1},f(x_{1},\cdots,x_{d-1})-r_{0}),\\
 \gamma_{r}(t,x_{1},\cdots,x_{d-1})&:=\emph{Tr}^{+}(uP,(0,T) \times \Sigma_{r})(t,x_{1},\cdots,x_{d-1},f(x_{1},\cdots,x_{d-1})-r) .
 \end{aligned}
 \label{space1}
 \end{equation}
 Then $\gamma_{r} \rightarrow \gamma_{0}$ strongly in $L^{1}((0,T)\times D) $ as $r \rightarrow r^{+}_{0} $.
 \end{theorem}
The proof of the above result follows the same strategy as the proof of~\cite[Proposition 3.5]{CDS1} and is therefore omitted.
 \section{Applications to the Keyfitz and Kranzer system}
 \label{s:KK}
 In this section, we consider the initial-boundary value problem for the Keyfitz and Kranzer system~\cite{KK} of conservation laws in several space dimensions, namely 
 \begin{equation}
 \left\{
 \begin{array}{lll}
 \partial_{t} U+\displaystyle{ 
 \sum_{i=1}^{d} \partial_{x_{i}} (f^{i}(\vert U \vert) U)=0}
 & \text{in}\ (0,T) \times \Omega  \\
 U = U_{0} &  \text{at $t=0$} \\
 U = U_{b} & \text{on} \ \Gamma.
 \displaystyle{\phantom{\int}}  \\
 \end{array}
 \right.
 \label{KK1}
 \end{equation}
 Note that, in general, we cannot expect that the boundary datum is pointwise attained on the whole boundary $\Gamma$. We come back to this point in the following. 
  
 We follow the same approach as in~\cite{ABD,AD,Br,Delellis2} and we formally split the equation at the first line of~\eqref{KK1} as the coupling between a scalar conservation law and a linear transport equation. More precisely, we set  $F:=(f^{1},\cdots,f^{d})$ and we point out that the modulus 
 $\rho: = |U|$ formally solves the initial-boundary value problem 
 \begin{equation}
 \left\{
 \begin{array}{ll}
 \partial_{t} \rho+\text{div} (F(\rho) \rho)=0 &  \text{in}\ (0,T) \times \Omega\\
 \rho =\vert U_{0} \vert & \text{at $t=0$}\\
 \rho= \vert U_{b} \vert & \text{on}\, \Gamma.
 \end{array}
 \right.
 \label{KK2}
 \end{equation}
 We follow~\cite{BLN,CR,Serre2} and we extend  notion of \emph{entropy admissible} solution (see~~\cite{Kr}) to initial boundary value problems. 
 \begin{definition}
 A function $\rho \in L^{\infty}((0,T) \times \Omega) \cap BV((0,T) \times \Omega) $ is  an entropy solution of \eqref{KK2} if for all $k \in \mathbb{R}$,  
 \begin{equation}
 \begin{aligned}
 &\int_{0}^{T} \int_{\Omega} \Big\{\vert \rho(t,x)-k \vert\ \partial_{t} \psi + \emph{sgn}(\rho-k)[F(\rho)-F(k)] \cdot \nabla \psi \Big\} \ dx dt \\
 &+ \int_{\Omega} \vert \rho_{0}-k \vert\ \psi(0, \cdot) \ dx -\int_{0}^{T} \int_{\partial \Omega} \emph{sgn}(\vert U_{b} \vert(t,x)-k)\ [F(T(\rho))-F(k)]\cdot \vec n \ \psi \ dx dt \geq 0, 
 \end{aligned}
  \notag
 \end{equation}
 for any positive test function $\psi \in C^{\infty}_{c}([0,T) \times \mathbb{R}^{d}; \mathbb{R}^{+}).$ In the above expression $T(\rho)$ denotes the trace of the function $\rho$ on the boundary $\Gamma$ and $\vec n$ is the outward pointing, unit normal vector to $\Gamma$.
 \end{definition}
 Existence and uniqueness results for entropy admissible solutions of the above systems were obtained by Bardos, le Roux and N{\'e}d{\'e}lec~\cite{BLN} by extending the analysis by Kru{\v{z}}kov to initial-boundary value problems (see also~\cite{CR,Serre2} for a more recent  
 discussion). Note, however, that one cannot expect that the boundary value $|U_b|$ is pointwise attained on the whole boundary $\Gamma$, see again~\cite{BLN,CR,Serre2} for a more extended discussion. 
 
Next, we introduce the equation for the \emph{angular part} of the solution of~\eqref{KK1}. We recall that, if $|U_b|$ and $|U_0|$ are of bounded variation, then so is $\rho$ and hence the trace of $F(\rho) \rho$ on $\Gamma$ is well defined. As usual, we denote it by $T(F(\rho) \rho)$. In particular, we can introduce  the set 
$$
   \Gamma^- : = \big\{ (t, x) \in \Gamma: \; T(F(\rho) \rho) \cdot \vec n <0 \big\}, 
$$
 where as usual $\vec n$ denotes the outward pointing, unit normal vector to $\Gamma$. We consider the vector $\theta=(\theta_{1},\cdots,\theta_{N}) $ and we impose 
  \begin{equation}
 \left\{
 \begin{array}{llll}
 \partial_{t}(\rho \theta)+\text{div}(F(\rho) \rho \theta)=0 & \text{in}\ (0,T) \times \Omega \phantom{\displaystyle{\int}}\\
 \theta=\displaystyle{\frac{U_{0}}{\vert U_{0} \vert}}& \text{at $t=0$}\\
 \theta=\displaystyle{\frac{U_{b}}{\vert U_{b} \vert}} & \text{on} \ \Gamma^{-},
 \end{array}
 \label{KK5}
 \right.
 \end{equation}
where the ratios $U_0 / |U_0|$ and $U_b / |U_b|$ are defined to be an arbitrary unit vector when $|U_0|=0$ and $|U_b|=0$, respectively. Note that the product $U=\theta \rho$ formally satisfies the equation at the first line of~\eqref{KK1}. We now extend the notion of \emph{renormalized entropy solution} given in~\cite{ABD,AD,Delellis2} to initial-boundary value problems.  
 \begin{definition}
 \label{d:res}
 A renormalized entropy solution of~\eqref{KK1} is a function $U \in L^\infty ( (0, T) \times \Omega; \mathbb R^N)$ such that $U = \rho \theta$, where 
 \begin{itemize}
 \item $\rho = |U|$ and $\rho$ is an entropy admissible solution of~\eqref{KK2}.
 \item $\theta = (\theta_1, \dots, \theta_N)$ is a distributional solution, in the sense of Definition~\ref{d:distrsol}, of~\eqref{KK5}.  
 \end{itemize}
 \end{definition}
 Some remarks are here in order. First, we can repeat the proof of \cite[Proposition 5.7]{Delellis1}  and conclude that, under fairly general assumptions, any renormalized entropy solution is an entropy solution. More precisely, let us fix a renormalized entropy solution $U$ and an \emph{entropy-entropy flux pair} $(\eta, Q)$, namely a couple of functions $\eta: \R^N \to \R$, $Q: \R^N \to \R^d$ such that   
 $$
     \nabla \eta D f^i = \nabla Q^i, \quad \text{for every $i=1, \dots, d$.}
 $$
 Assume that 
 $$
     \mathcal L^1 \big\{ r \in \R: \; (f^1)'(r) = \dots = (f^d)' (r)=0 \big\}=0. $$
 By arguing as in~\cite{Delellis1} we conclude that, if $\eta$ is convex, then 
$$
     \int_0^T \int_\Omega \eta (U) \partial_t \phi + Q(U)\cdot  \nabla \phi \, dx dt \ge 0
 $$
 for every \emph{entropy-entropy flux pair} $(\eta, Q)$ and for every nonnegative test function $\phi \in C^\infty_c ((0, T) \times \Omega)$. 
  
Second, we point out that, as the  Bardos, le Roux and N{\'e}d{\'e}lec~\cite{BLN} solutions of scalar initial-boundary value problems, renormalized entropy solutions of the Keyfitz and Kranzer system do not, in general pointwise attain the boundary datum $U_0$ on the whole boundary $\Gamma$. 
 
 We now state our well-posedness result. 
 \begin{theorem}
 \label{t:KK} Assume $\Omega$ is a bounded open set with $C^2$ boundary. Also, assume that $U_0 \in L^\infty (\Omega; \mathbb R^N)$ and $U_b \in L^\infty (\Gamma; \mathbb R^N)$ satisfy $|U_0| \in 
 BV ( \Omega)$, $|U_b| \in BV (\Gamma).$ Then there is a unique renormalized entropy solution of~\eqref{KK1} that satisfies $U \in L^\infty ((0, T)\times \Omega; \mathbb R^N)$. 
 \end{theorem}
 \begin{proof} We first establish existence, next uniqueness. \\
 {\sc Existence:} first, we point out that the results in~\cite{BLN,CR,Serre2} imply that there is an entropy admissible solution of~\eqref{KK2} satisfying
 $\rho \in L^\infty ((0, T) \times \Omega) \cap BV ((0, T) \times \Omega).$ 
 Also, $\rho$ satisfies the maximum principle, namely
 \begin{equation}
 \label{e:rhomaxprin}
     0 \leq \rho \leq \max \big\{ \| U_0 \|_{L^\infty}, \|U_b \|_{L^\infty} \big\}. 
 \end{equation}
 For every $j=1, \dots, N$ we consider the initial-boundary value problem 
  \begin{equation}
 \left\{
 \begin{array}{llll}
 \partial_{t}(\rho \theta_j)+\text{div}(F(\rho) \rho \theta_j)=0 & \text{in}\ (0,T) \times \Omega \phantom{\displaystyle{\int}}\\
 \theta_j=\displaystyle{\frac{U_{0j}}{\vert U_{0} \vert}}& \text{at $t=0$}\\
 \theta_j=\displaystyle{\frac{U_{bj}}{\vert U_{b} \vert}} & \text{on} \ \Gamma^{-},
 \end{array}
 \label{KK6}
 \right.
 \end{equation}
where $U_{0j}$ and $U_{bj}$ is the $j$-th component of $U_0$ and $U_b$, respectively. The existence of a distributional solution $\theta_j$ follows from the existence part in Theorem~\ref{IBVP-NC}. 

We now set $U: = \rho \theta$, where $\theta = (\theta_1, \dots, \theta_N)$. To conclude the existence part we are left to show that $|U|=\rho$. To this end, we point out that, by combining~\cite[Lemma 5.10]{Delellis1} (renormalization property inside the domain) with Theorem~\ref{trace-renorm} (trace renormalization property) and by arguing as in \S~\ref{s:uni}, we conclude that, for every $j=1, \dots, N$, $\theta^2_j$ is a distributional solution, in the sense of Definition~\ref{d:distrsol}, of the initial-boundary value problem 
  \begin{equation*}
 \left\{
 \begin{array}{llll}
 \partial_{t}(\rho \theta^2_j)+\text{div}(F(\rho) \rho \theta^2_j)=0 & \text{in}\ (0,T) \times \Omega \phantom{\displaystyle{\int}}\\
 \theta_j=\displaystyle{\frac{U^2_{0j}}{\vert U_{0} \vert^2}}& \text{at $t=0$}\\
 \theta=\displaystyle{\frac{U^2_{bj}}{\vert U_{b} \vert^2}} & \text{on} \ \Gamma^{-}.
 \end{array}
 \right.
 \end{equation*}
By adding from $1$ to $N$, we conclude that $|\theta|^2$ is a distributional solution of 
\begin{equation*}
 \left\{
 \begin{array}{llll}
 \partial_{t}(\rho |\theta|^2)+\text{div}(F(\rho) \rho |\theta|^2)=0 & \text{in}\ (0,T) \times \Omega \\
 \theta_j=1& \text{at $t=0$}\\
 \theta=1 & \text{on} \ \Gamma^{-}.
 \end{array}
 \right.
 \end{equation*}
 By recalling the equation at the first line of~\eqref{KK2} we infer that $|\theta|^2 =1$ is a solution of the above initial-boundary value problem. By the uniqueness part of Theorem~\ref{IBVP-NC}, we then deduce that $\rho |\theta|^2= \rho$ and this concludes the proof of the existence part. \\
 {\sc Uniqueness:} assume $U_1$ and $U_2$ are two renormalized entropy solutions, in the sense of Definition~\ref{d:res}, of the initial-boundary value 
problem~\eqref{KK1}. Then $\rho_1: = |U_1|$ and $\rho_2 : =|U_2|$ are two 
entropy admissible solutions of the initial-boundary value problem~\eqref{KK2} and hence $\rho_1=\rho_2$. By applying the uniqueness part of Theorem~\ref{IBVP-NC} to the initial-boundary value problem~\eqref{KK6}, for every $j=1, \dots, N$, we can then conclude that $U_1 =U_2$. 
  \end{proof}

 \section*{Acknowledgments}
This paper has been written while APC was a postdoctoral fellow at the University of Basel supported by a ``Swiss Government Excellence Scholarship'' funded by the State Secretariat for Education, Research and Innovation (SERI). APC would like to thank the SERI for the support and the Department of Mathematics and Computer Science of the University of Basel for the kind hospitality. GC was partially supported by the Swiss National Science Foundation (Grant 156112). LVS is a member of the GNAMPA group of INdAM (``Istituto Nazionale di Alta Matematica"). Also, she would like to thank the Department of Mathematics and Computer Science of the University of Basel for the kind hospitality during her visit, during which part of this work was done.



\small 
\begin{thebibliography}{10}
\bibitem{A} L. Ambrosio, Transport equation and Cauchy problem for BV vector fields, \textit{Invent. Math.} 158 (2004) 227-260.

\bibitem{ABD} L. Ambrosio, F. Bouchut, C. De Lellis, Well-posedness for a class of hyperbolic systems of conservation laws in several space dimensions, \textit{Commun. Partial Differ. Equ.} 29 (2004), no. 9-10, 
1635-1651. 

\bibitem{AC} L. Ambrosio, G. Crippa, Continuity equations and ODE flows with non-smooth velocity, \textit{Proc. R. Soc. Edinb., Sect. A, Math.}, no. 6 (2014), 1191-1244.  

\bibitem{ACFS} L. Ambrosio, G. Crippa, A. Figalli, L.V. Spinolo, Some new well-posedness results for continuity and transport equations, and applications to the chromatography system, \textit{SIAM J. Math. Anal.} 41 (2009), no. 5, 1890-1920. 

\bibitem{ACM} L. Ambrosio, G. Crippa, S. Maniglia, Traces and fine properties of a BD class of vector fields and applications, \textit{Ann. Fac. Sci. Toulouse} (6) 14 (2005) 527-561. 

\bibitem{AD} L. Ambrosio, C. De Lellis, Existence of solutions for a class of hyperbolic systems of conservation laws in several space dimensions, \textit{Int. Math. Res. Not.} 41 (2003), 2205-2220. 

\bibitem{AFP} L. Ambrosio, N. Fusco, D. Pallara, Functions of Bounded Variation and Free Discontinuity Problems, Oxford Math. Monogr., The Clarendon Press Oxford University Press, New York, 2000. 

\bibitem{Anz} G. Anzellotti, Pairings between measures and bounded functions and compensated compactness,
\textit{Ann. Mat. Pura. Appl.} (4) 135 (1983) 293-318.

\bibitem{Bar} C. Bardos, Probl{\`e}mes aux limites pour les {\'e}quations aux d{\'e}riv{\'e}es partielles du premier ordre {\`a} coefficients r{\'e}els; th{\'e}or{\`e}mes d'approximation; application {\`a} l'{\'e}quation de transport, \textit{Ann. Sci. {\'E}c. Norm. Sup{\'e}r.} (4) 3 (1970) 185-233.

\bibitem{BLN} C. Bardos, A-Y. Le Roux, J.C. Nedelec, First order quasilinear equations with boundary conditions, \textit{Commun. Partial Differ. Equ.} 4 (1979) 1017-1034.

\bibitem{Boyer} F. Boyer, Trace theorems and spatial continuity properties for the solutions of the transport equation, \textit{Differ. Integral Equ.} 18 (2005) 891-934.

\bibitem{Br} A. Bressan, An ill-posed Cauchy problem for a hyperbolic system in two space dimensions, \textit{Rend. Semin. Mat. Univ. Padova} 110 (2003) 103-117.

\bibitem{CF} G.-Q. Chen, H. Frid, Divergence-measure fields and hyperbolic conservation laws, \textit{Arch. Ration. Mech. Anal.} 147 (1999) 89-118.

\bibitem{CTZ} G.-Q. Chen, M. Torres, W.P. Ziemer, Gauss-Green theorem for weakly differentiable vector fields, sets of finite perimeter, and balance laws, \textit{Commun. Pure. Appl. Math.} 62 (2009) 242-304.

\bibitem{CR} R.M. Colombo, E. Rossi, Rigorous estimates on balance laws in bounded domains, \textit{Acta Math. Sci., Ser. B, Engl. Ed.} 35 (2015), no. 4, 906-944.

\bibitem{CDS1} G. Crippa, C. Donadello, L.V. Spinolo, Initial-boundary value problems for continuity equations with BV coefficients,
\textit{J. Math. Pures Appl.} (9) 102, no. 1, 79-98 (2014). 


\bibitem{CDS2} G. Crippa, C. Donadello, L.V. Spinolo, A note on the initial-boundary value problem for continuity equations with rough 
coefficients, \textit{Hyperbolic problems: theory, numerics and applications}, AIMS Series on Appl. Math. 8 (2014) 957-966.


\bibitem{Daf} C.M. Dafermos, Hyperbolic Conservation laws in Continuum Physics, third ed., Grundlehren Math. Wiss., vol. 325, Springer-Verlag, Berlin, 2010.

\bibitem{Delellis1} C. De Lellis, Notes on hyperbolic systems of conservation laws and transport equations, in: Handbook of Differential
Equations: Evolutionary Equations, vo. III, Elsevier/North-Holland, Amsterdam, 2007, pp. 277-382.

\bibitem{Delellis2} C. De Lellis, Ordinary differential equations with rough coefficients and the renormalization theorem of
Ambrosio [after Ambrosio, DiPerna, Lions], Asterisque (2008) 175-203 (Exp. No. 972). Seminaire Bourbaki, vol. 2006/2007.

\bibitem{Diperna-Lions} R.J. DiPerna, P.L. Lions, Ordinary differential equations, transport theory and Sobolev spaces, \textit{Invent. Math.} 98 (1989) 511-547. 

\bibitem{GS} V. Girault, L.R. Scott, On a time-dependent transport equation in a Lipschitz domain, \textit{SIAM J. Math. Anal. } 42 (2010) 1721-1731.

{\bibitem{GR} E. Godlewski, P.A. Raviart, Numerical approximation of hyperbolic systems of conservation laws, \textit{Applied Mathematical Sciences.} 118. New York, NY: Springer. viii, 509 p. (1996). }
  

\bibitem{KK} B.L. Keyfitz, H.C. Kranzer, A system of non-strictly hyperbolic conservation laws arising in elasticity theory, \textit{Arch. Ration. Mech. Anal.} 72 (1980) 219-241.

\bibitem{Kr} S.N. Kruzhkov, First order quasilinear equations in several independent variables, \textit{Math. USSR, Sb.} 10 (1970) 217-243. 


\bibitem{Mis} S. Mischler, On the trace problem for solutions of the Vlasov equation, \textit{Commun. Partial Differ. Equ.} 25 (2000) 1415-1443. 

\bibitem{Serre1} D. Serre, \textit{Systems of Conservation Laws, I. Hyperbolicity, Entropies, Shock Waves}, Cambridge
University Press, Cambridge, 1999. 

\bibitem{Serre2} D. Serre, \textit{Systems of Conservation Laws, II. Geometric Structures, Oscillations, and Initial-Boundary Value Problems}, Cambridge University Press, Cambridge, 2000.

  

\end{thebibliography}
\end{document}